\documentclass[11pt]{amsart}
\usepackage{amsxtra}
\usepackage{graphicx}
\usepackage{amssymb,mathabx}
\usepackage{color}
\usepackage[all]{xy}
\usepackage{comment}

\usepackage{multicol}
\usepackage{tabularx}

\theoremstyle{plain}

\newcommand{\cleqn}{\setcounter{equation}{0}}
\newcommand{\clth}{\setcounter{theorem}{0}}
\newcommand {\sectionnew}[1]{\section{#1}\cleqn\clth}
\newcommand{\nn}{\hfill\nonumber}
\newtheorem{theorem}{Theorem}[section]
\newtheorem{lemma}[theorem]{Lemma}
\newtheorem{definition-theorem}[theorem]{Definition-Theorem}
\newtheorem{proposition}[theorem]{Proposition}
\newtheorem{corollary}[theorem]{Corollary}
\newtheorem{definition}[theorem]{Definition}
\newtheorem{example}[theorem]{Example}

\newtheorem{remark}[theorem]{Remark}
\newtheorem{conjecture}[theorem]{Conjecture}
\newtheorem{notation}[theorem]{Notation}

\newcommand \bth[1] { \begin{theorem}\label{#1} }
\newcommand \ble[1] { \begin{lemma}\label{#1} }

\newcommand \bpr[1] { \begin{proposition}\label{#1} }
\newcommand \bco[1] { \begin{corollary}\label{#1} }
\newcommand \bde[1] { \begin{definition}\label{#1}\rm }
\newcommand \bex[1] { \begin{example}\label{#1}\rm }
\newcommand \bre[1] { \begin{remark}\label{#1}\rm }
\newcommand \bcj[1] { \begin{conjecture}\label{#1}\rm }

\newcommand \bnota[1] { \begin{notation}\label{#1}\rm }
\renewcommand {\eth} { \end{theorem} }
\newcommand {\ele} { \end{lemma} }

\newcommand {\epr} { \end{proposition} }
\newcommand {\eco} { \end{corollary} }
\newcommand {\ede} { \end{definition} }
\newcommand {\eex} { \end{example} }
\newcommand {\ere} { \end{remark} }
\newcommand {\ecj} { \end{conjecture} }

\newcommand {\enota} { \end{notation} }

\def \Cset {{\mathbb C}}

\def \Zset {{\mathbb Z}}
\def \Aset {{\mathbb A}}
\def \Nset {{\mathbb N}}

\def \Pset {{\mathbb P}}

\def \cp   {{\Cset \Pset^1}}

\def \OO {{\mathcal{O}}}

\def \mt  {\mapsto}

\def \ol {\overline}
\def \wt {\widetilde}

\def \id { {\mathrm{id}} }

\def \Ad { {\mathrm{Ad}} }

\def \Lie { {\mathrm{Lie \,}} }

\def \g  {\mathfrak{g}}   

\def \sl {\mathfrak{sl}} 
\def \h  {\mathfrak{h}}

\def \l {\mathfrak{l}}
\def \sl {\mathfrak{sl}}

\DeclareMathOperator \GCR { {\mathcal{ECR}} }
\DeclareMathOperator \Span { {\mathrm{Span}} }
\DeclareMathOperator \Aut { {\mathrm{Aut}} }
\DeclareMathOperator \redd { {\mathrm{red}}}

\DeclareMathOperator \ad { {\mathrm{ad}} }

\DeclareMathOperator \Ker { {\mathrm{Ker}} }

\renewcommand \max { {\mathrm{max}} }

\newcommand{\Rvw}{\overline{\mathcal{R}_v^w}}
\newcommand{\oRvw}{\mathcal{R}_v^w}

\begin{document}
\title
{The Poisson degeneracy locus of a Flag Variety}
\author[\'Elie Casbi]{\'Elie Casbi}
\address{
Department of Mathematics \\
Northeastern University \\
Boston, MA 02115 \\
U.S.A.
}
\email{e.casbi@northeastern.edu, masoomi.a@northeastern.edu, m.yakimov@northeastern.edu}
\author[Aria Masoomi]{Aria Masoomi}
\author[Milen Yakimov]{Milen Yakimov}
\date{}
\begin{abstract} We present a comprehensive study of the degeneracy loci of the full flag varieties of all complex semisimple Lie groups equipped with the standard Poisson structures. The reduced Poisson degeneracy loci are shown to stratify under the action of the canonical maximal torus into open Richardson varieties $\mathcal{R}_v^w$ for pairs of Weyl group elements $v \leq w$ that extend the covering relation of the Bruhat order. Four different combinatorial descriptions of those pairs are given, and it is shown that their Bruhat intervals are power sets. The corresponding closed Richardson varieties $\overline{\mathcal{R}_v^m}$ are shown to be isomorphic to $(\mathbb{C}\mathbb{P}^1)^d$ for $d \geq 0$ in a  compatible way with the stratification. As a consequence, we obtain that the reduced Poisson degeneracy loci of all full flag varieties are connected, and all of their irreducible components are isomorphic to $(\mathbb{C}\mathbb{P}^1)^n$ for some $n \geq 0$; they are not equidimensional in general. Using the framework of projected Richardson varieties, these results are extended to all partial flag varieties. The top dimension of irreducible components of the reduced Poisson degeneracy locus in the full flag case 
is proved to be equal to 
the cardinality of Kostant's cascade of roots and the reflective length of the longest Weyl group element. It is shown that the Poisson degeneracy loci of flag varieties are not reduced in general. 
\end{abstract}
\maketitle
\sectionnew{Introduction}
\label{sec:intro}
\subsection{The degeneracy loci of Poisson schemes}
Let $X$ be a complex Poisson scheme, $\OO_X$ be its structure sheaf and $\Omega_X:=\Omega_{X/\Cset}$ be its sheaf of K\"ahler differentials. The Poisson bracket on the structure sheaf 
gives rise to a morphism of $\OO_X$-modules 
\[
\pi : \wedge^2 \Omega_X \to \OO_X.
\]
The {\em{$k$-th Poisson degeneracy locus}} $D_{2k}(X)$ of $X$ is the closed subcheme whose ideal sheaf is the image of the corresponding 
morphism of $\OO_X$-modules
\begin{equation}
\label{pik}
\pi^{k+1} : \wedge^{2k +2} \Omega_X \to \OO_X.
\end{equation}
If $X$ is smooth, then the Poisson structure is given by a bivector field 
\[
\pi \in H^0(X, \wedge^2 T_X)
\]
and the map \eqref{pik} is given by the contraction of $\wedge^{k+1} \pi$ with differential forms. In this situation, we will denote by $D_{2k} (X, \pi)$
the $2k$-th degeneracy locus and by $D_{2k}(X, \pi)_{\redd}$ the corresponding reduced subscheme of $X$, which we will refer to as to the 
{\em{reduced Poisson degeneracy locus}}. The latter equals the union of symplectic leaves of $(X, \pi)$ of dimension $\leq 2k$. 
Both $D_{2k}(X, \pi)$ and $D_{2k}(X, \pi)_{\redd}$ are Poisson subschemes of $(X, \pi)$, \cite{P}; 
we refer the reader to \cite{P} and \cite[Sect. 2]{GP} for background on Poisson schemes and to \cite{Pym} for complex analytic aspects of Poisson geometry. We will call $D_0(X, \pi)$ and 
$D_0(X, \pi)_{\redd}$ simply the degeneracy locus and the reduced Poisson degeneracy locus of $(X, \pi)$. 

Bondal conjectured that for all Fano Poisson varieties $(X, \pi)$, $D_{2k}(X, \pi)_{\redd}$ is nonempty and has a component of dimension $\geq 2k+1$ for all $k < \dim X/2$. 
Polishchuk proved that, if generically 
$\pi$ has rank $2n$, then the conjecture holds for $k=n-1$, see also \cite{Beauville}. Gualtieri and Pym \cite{GP} proved the conjecture for 4-folds, and more generally showed that, if $\dim X = 2n$, then either 
$D_{2n-2}(X, \pi)_{\redd}= X$ or is a (nonempty) hypersurface and that Bondal's conjecture holds for $k =n-2$. 
\subsection{Poisson homogeneous spaces, symplectic leaves and Poisson degeneracy loci}
Currently, there is little known about the exact structure of degeneracy loci of Poisson varieties. 
The Poisson degeneracy loci of the Fe\u{\i}gin--Odesski\u{\i} Poisson structures \cite{FO} on projective spaces were related 
to secant varieties and were conjectured to be reduced \cite[Sect. 8]{GP}. 
Results of this sort about the explicit form of $D_{2k}(X, \pi)$ for concrete Poisson varieties $(X, \pi)$ 
(rather than just dimension formulas for irreducible components) are rare.  

The goal of this paper is to initiate a study of the explicit structure of the degeneracy loci of Poisson varieties that play a 
role in Lie theory and quantum groups. After the fundamental work of Drinfeld \cite{D}, {\em{Poisson homogeneous spaces}}
developed into a fundamental approach to the representation theory of quantum groups and their actions on algebras. These are homogenous 
spaces $M$ of Lie groups $G$ equipped with Poisson structures $\pi$ and $\Pi$, respectively, such that the maps 
\[
(G, \Pi) \times (G, \Pi) \to (G, \Pi) 
\quad \mbox{and} \quad (G, \Pi) \times (M, \pi) \to (M, \pi)
\]
are Poisson; the first condition makes the pair $(G, \pi)$ a Poisson--Lie group and the second makes $(M, \pi)$ 
a Poisson homogeneous space of it. Drinfeld \cite{D} proved that Poisson homogeneous spaces are classified in terms 
of Lagrangian subalgebras of certain quadratic Lie algebras. Evens and Lu \cite{EL1,EL2} used this to show that Poisson homogeneous 
spaces have natural embeddings in varieties of Lagrangian subalgebras, which allows for the former to be studied in families 
from a unified stand point. A general approach to the classification of symplectic leaves in Poisson--Lie groups and Poisson homogeneous 
spaces via dressing orbits was found by Semenov-Tian-Shansky \cite{STS} and Karolinsky \cite{Karolinsky}. Symplectic foliations were explicitly 
described for all Belavin--Drinfeld Poisson structures \cite{Y-Duke}, standard Poisson structures 
on symmetric spaces and flag varieties \cite{BGY,FL,GY}, wonderful compactifications \cite{EL2}, and other situations. 
A general method for constructing partitions into regular Poisson submanifolds that encompasses all of the above classes was developed in 
\cite{LY}. 

However, these sets of results only determine Poisson degeneracy loci set theoretically. 
They do not say much about the geometry of these subschemes. 
\subsection{Results} In this paper we carry out a comprehensive study of the (0-th) Poisson degeneracy locus of the simplest and most important 
example of Fano varieties in the list of Poisson homogeneous spaces, namely the full flag varieties $G/B_+$ 
of all complex semisimple Lie groups $G$ equipped with the standard Poisson structure $\pi$ associated to a choice of opposite 
Borel subgroups $B_\pm$, cf. \eqref{eq:pi}. Our first main result is as follows:
\medskip

\noindent
{\bf{Theorem A.}} {\em{For every complex semisimple Lie group $G$, the reduced Poisson degeneracy locus $D_0(G/B_+, \pi)_{\redd}$ is stratified into the disjoint union of all open Richardson varieties 
\[
\oRvw = B_- v B_+/B_+ \cap B_+ w B_+/B_+ \subseteq G/B_+
\] 
for pairs of Weyl group elements $v \leq w \in W$ with $d := l(w) - l(v)$ satisfying any of the following four equivalent conditions:}}
\begin{enumerate}
    \item {\em{$\dim\Ker(vw^{-1}+1) = l(w) - l(v)$.}}
    \item {\em{$(v w^{-1})^2 =1$ and $l_\Delta(v w^{-1}) = l(w) - l(v)$ where $l_\Delta$ denotes the reflection length on the Weyl group of $G$, see Section \ref{sec:background-Lie}.}} 
    \item {\em{$v = s_{\gamma_1} \dots s_{\gamma_d} w$ for some positive roots $\gamma_1, \ldots, \gamma_d \in \Delta_+$ such that $\gamma_j \bot \gamma_k, \forall j \neq k$.}}
    \item {\em{For one, and thus for each reduced word $(i_1, \ldots, i_{l(w)})$ of $w$, there exits a reduced subword $(i_1, \ldots , \widehat{i}_{p_1} , \ldots , \widehat{i}_{p_d} , \ldots , i_l)$ with value $v$ 
    such that $\beta_{p_{j}} \bot \beta_{p_{k}}$, $\forall j \neq k$ for the roots $\{\beta_1, \ldots, \beta_{l(w)} \} = \Delta_+ \cap w(- \Delta_+)$ given by \eqref{eq:beta-k}.}}
\end{enumerate}
\medskip

Denote by $\GCR(W)$ the pairs of Weyl group elements $v \leq w$ satisfying any of the four equivalent conditions in the theorem. This set is an extension of the covering relation $\lessdot$ 
of the Bruhat order on $W$ in the sense that 
\[
(w, w) \in \GCR(W) \quad \mbox{and} \quad (v,w) \in \GCR(W), \; \; \forall v \lessdot w.
\]
Condition (4) in Theorem A appears in the work of Heckeberger--Kolb \cite{HK} on the classification of characters of the quantized coordinate rings \cite{DKP} of the Schubert cells of $G/B_+$ 
based on the classification of their prime ideals in \cite{Y-PLMS}. 

The Poisson structure $\pi$ is invariant under the maximal torus $T:=B_+ \cap B_-$. We prove that each $\oRvw$ is a single $T$-orbit for $(v,w) \in \GCR(W)$, i.e., {\em{the $T$-orbit stratification of the reduced Poisson degeneracy 
locus $D_0(G/B_+, \pi)_{\redd}$ is given by}}
\[
D_0(G/B_+, \pi)_{\redd} = \bigsqcup_{(v,w) \in \GCR(W)} \oRvw.
\] 
It is not true that the action of $T$ on $\oRvw$ is transitive (or equivalently that $\Rvw$ is a toric variety) only if $(v,w) \in \GCR(W)$, see Remark \ref{r:T-act}.

The pairs in $\GCR(W)$ have remarkable geometric and combinatorial properties obtained in our next main result:
\medskip

\noindent
{\bf{Theorem B.}} {\em{For all $(v,w) \in \GCR(W)$ the following hold:}}
\begin{enumerate}
    \item[(i)] {\em{The closed Richardson variety $\Rvw$ is isomorphic to a power of $\cp$
    \[
     \Rvw \simeq (\cp)^{l(w)-l(v)}.
     \] 
     }}
     \item[(ii)] {\em{The Bruhat interval $[v,w]$ is isomorphic to the power set of $\{1, \ldots, l(w) - l(v)\}$.}} 
    \item[(iii)] {\em{There exists only one Deodhar stratum \cite{Deodhar} in the open Richardson variety $\oRvw$, which is the unique open stratum $(\Cset^\times)^{l(w)-l(v)}$.}}
    \item[(iv)] {\em{The corresponding Kazhdan--Lusztig $R$-polynomial is 
    \[
    R_{v,w}(q) = (q-1)^{l(w)-l(v)}.
    \]
    }}
\end{enumerate}
\medskip

The special case of Theorem B(i) for the pairs of the form $(1,w) \in \GCR(W)$ was proved by Ohn in \cite{Ohn}.

The results in Theorem B are extended to the case of partial flag varieties in Theorem \ref{t:kappaP-isom}, where the statements are in terms of the Knutson--Lam--Speyer 
projected Richardson varieties \cite{KLS-C,KLS}. The results of Theorem A also hold for partial flag varieties due to Corollary \ref{c:GmodP}. 

Our third main result describes the geometry of the Poisson degeneracy loci of flag varieties:
\medskip

\noindent
{\bf{Theorem C.}} {\em{The following hold for all complex semisimple Lie groups $G$:}}
\begin{enumerate}
\item[(i)] {\em{The reduced Poisson degeneracy locus $D_0(G/B_+, \pi)_{\redd}$ is connected and its irreducible components are isomorphic to $(\cp)^{l(w) - l(v)}$ for the maximal elements of the poset $\GCR(W)$ 
with the partial order $(v', w') \leq (v, w)$ if $v \leq v' \leq w' \leq w$.}} 
\item[(ii)] {\em{The top-dimension of irreducible components of $D_0(G/B_+, \pi)_{\redd}$ is equal to 
\[
\sharp \mathcal{B} = l_{\Delta}(w_0), 
\]
where $\mathcal{B}$ is Kostant's cascade \cite{Kostant} of roots of $\Lie G$ and $l_\Delta(w_0)$ is the reflective length of the longest element $w_0$ of the Weyl group of $G$. The reduced 
Poisson degeneracy locus $D_0(G/B_+, \pi)_{\redd}$ is not equidimensional in general.}}  
\item[(iii)] {\em{The Poisson degeneracy locus $D_0(G/B_+, \pi)$ is not reduced in general.}} 
\end{enumerate}
\medskip

We note the difference between Theorem C(iii) and the results in \cite[Sect. 8]{GP}. While the degeneracy loci of the 
Fe\u{\i}gin--Odesski\u{\i} Poisson structures \cite{FO} on projective spaces were conjectured to be reduced \cite[Sect. 8]{GP}, we prove that those for flag varieties 
are not. In Section \ref{non-reducedness} a detailed picture for the full flag varieties of $SL_3(\Cset)$ and $SL_4(\Cset)$ is presented. 

In \cite{GY-MEMO} cluster algebra structures on arbitrary Poisson CGL extensions were constructed. 
For example, this construction applies to the coordinate rings of the Schubert cells of all flag varieties $G/B_+$ with the restrictions of the  
Poisson structure $\pi$. It is interesting to understand the relationship between the geometry of the degeneracy  
loci $D_0(G/B_+, \pi)$ and the structure of the cluster algebras in question. In the simply laced case the latter were much studied from many points of view, e.g. additive and monoidal categorifications \cite{GLS,KKKO}. 

Artin, Tate and Van den Bergh \cite{ATVdB} introduced the notion of point schemes for $\Nset$-graded connected algebras and proved 
important properties of regular algebras of dimension 3 using surjective maps to the associated twisted homogeneous coordinate rings (here and below $\Nset := \{0, 1, \ldots\}$).
Rogalski and Zhang \cite{RZ} proved that this map is surjective in large degrees for all strongly Noetherian algebras generated in degree 1. Point schemes can be thought of as 
quantum analogs of Poisson degeneracy loci. The results in this paper and the classification of prime ideals of quantum flag varieties \cite{Y-PAMS} form the basis of a description of the point schemes of quantum flag varieties \cite{LR,S} and the associated twisted homogeneous coordinate rings which we will give in a future publication. 
\medskip

\noindent
{\bf Acknowledgements.} The research of E.C. and A.M. was supported by NSF grant DMS–2200762 and the RTG NSF grant DMS-1645877. The research of M.Y. was supported by NSF grant DMS–2200762.  

\sectionnew{background material on semisimple Lie groups} \label{section : Lie theory}

 In this section we gather material on Weyl groups, flag varieties and Richardson varieties which will be needed in the paper.

 \subsection{Weyl groups and root systems} \label{subsection : W gp and roots}
 
Let $G$ be a connected simply connected complex semisimple Lie group. Denote by $\mathfrak{g} = \mathrm{Lie} \, G$ its Lie algebra. Let $B_{\pm}$ be fixed opposite Borel subgroups of $G$,  $T := B_{+}\cap B_{-}$ be the corresponding maximal torus of G and  let $ \mathfrak{b}_{\pm}$ and $ \mathfrak{t}$ denote their respective Lie algebras. We have $\mathfrak{g} = \mathfrak{b}_+ \oplus \mathfrak{t} \oplus \mathfrak{b}_-$. Denote by $P$ and $Q$ the weight and root lattices of $G$. Fix an index set 
$[1,n]:=\{1, \ldots, n\}$ of the vertices of the Dynkin diagram of $G$ and denote by 
\[
\Pi := \{ \alpha_i \mid 1 \leq i \leq n \} \subset \mathfrak{t}^{*}
\]
the set of simple roots, by $\Delta_{+}$ the set of positive roots and by $\Delta := \Delta_{+} \sqcup (- \Delta_{+})$ the set of roots of $G$. A subset $\Psi \subset \Delta$ is said to be convex if 
$$ \forall \mu, \nu \in \Psi, \mu + \nu \in \Delta \Rightarrow \mu + \nu \in \Psi. $$
Let $C=(c_{ij})_{i,j=1}^n$ be the Cartan matrix of $G$
and $d_1, \ldots, d_n$ be the positive integers symmetrizing $C$ such that the subcollection of those corresponding to any simple factor of $G$ are relatively prime. Let $( \cdot , \cdot)$ be symmetric non-degenerate bilinear form on $\mathfrak{t}^{*}$ such that 
\[
(\alpha_i, \alpha_j) = d_i c_{ij}.
\]
Set
\[
\| \gamma \|^2 := ( \gamma, \gamma ).
\]
We will denote by $<$ the partial order on $\Delta_+$ given by 
$$\beta < \gamma \Leftrightarrow \gamma - \beta \in \sum_i \Nset \alpha_i .$$
 For each $\beta \in \Delta_+$ there is an injective group homomorphism $u_{\beta} : \mathbb{G}_a \rightarrow G$ satisfying $tu_{\beta}(z)t^{-1} = u_{\beta}(\beta(t)z)$ for any $t \in T$ and $z \in \mathbb{C}$. Denote by $U_{\beta}$ the one-parameter unipotent subgroup of $G$ given by $U_{\beta} := u_{\beta}(\mathbb{G}_a)$. For each $\beta \in \Delta$ there is a unique injective homomorphism $\varphi_{\beta} : SL_2(\mathbb{C})  \rightarrow G$ given by 
$$ \varphi_{\beta} \begin{pmatrix} 1 & z \\ 0 & 1 \end{pmatrix}  = u_{\beta}(z), \qquad \varphi_{\beta} \begin{pmatrix} 1 & 0 \\ z & 1 \end{pmatrix}  = u_{-\beta}(z)  $$
for any $z \in \mathbb{C}$. The image of $\varphi_\beta$ will be denoted by $L_{\beta}$. Let
$$ \beta^{\vee}(z) :=  \varphi_{\beta} \begin{pmatrix} z & 0 \\ 0 & z^{-1} \end{pmatrix}  \in T $$
for $z \in \mathbb{C}^{\times}$. 
Denote the Borel subgroups of $L_\beta$ 
\[
B_{\pm \beta} := L_\beta \cap T U_{\pm \beta} 
\]
and the maximal torus
\[
T_\beta := L_\beta \cap T = B_{\beta} \cap B_{-\beta}.
\]

Denote by $W := N_G(T)/T$ the Weyl group of $G$, where $N_G(T)$ stands for the normalizer of $T$ in $G$. We have that 
$$  \dot{s}_{\beta} := \varphi_{\beta} \begin{pmatrix} 0 & 1 \\ -1 & 0 \end{pmatrix}  \in N_G(T). $$
Let $s_{\beta}$ be the class of $\dot{s}_{\beta}$ in $W$, the reflection associated to $\beta$. These reflections generate the group $W$, and a minimal collection of generators is given by the simple reflections, i.e., the reflections associated to the simple roots $\alpha_i, 1 \leq i \leq n$. Set for brevity
\[
s_i := s_{\alpha_i}.
\]
 We will make use of the following identity that holds for every $\beta \in \Delta$:
\begin{equation} \label{eqn: important identity}
 \forall z \in \mathbb{C}^{\times}, \quad  u_{- \beta}(z^{-1}) = u_{\beta}(z) \beta^{\vee}(z) \dot{s}_{\beta} u_{\beta}(z) .
\end{equation}
The Weyl group $W$ acts on the set of characters of $T$ by permuting the elements of $\Delta$ and hence can be viewed as a subgroup of $GL(\mathfrak{t}^{*})$, namely, the subgroup of $GL(\mathfrak{t}^{*})$ generated by the reflections $s_{\beta}$, $\beta \in \Delta$, given by 
$$ s_{\beta} : x \mapsto x - 2   \frac{(x,\beta)}{\|\beta \|^2} \beta . $$
Note that two reflections $s_{\beta}$ and $s_{\gamma}$ commute if and only if $\beta$ and $\gamma$ are orthogonal roots. 
Moreover we have that $v^{-1}U_{\beta}v = U_{v^{-1}(\beta)}$ for each Weyl group element $v \in W$ and $\beta \in \Delta$.

We identify $\mathfrak{t}^* \simeq \mathfrak{t}$ as vector spaces via the nondegenerate form $(.,.)$ on $\mathfrak{t}^*$ and use the identification to transfer the form to $\mathfrak{t}$. The last form has a unique 
extension to a nondegenerate invariant symmetric
bilinear form on $\mathfrak{g}$ which will be 
denoted by the same notation. 

For $w \in W$, we will denote by $\dot{w}$ a representative of it in the normalizer $N(T)$ of $T$. 
 \subsection{Length, reflection length and reduced expressions}
 \label{sec:background-Lie}

Given a generating set $R$ of $\mathfrak{t}^{*}$, the reflection length of an element $w \in W$ with respect to $R$, denoted $l_R(w)$, is the minimal $r \in \Nset$  such that $w$ can be written as a product of $r$ reflections $s_{\beta_1}, \ldots , s_{\beta_r}$ with $\beta_1, \ldots , \beta_r \in R$. In the case $R= \Pi$, $l_{\Pi}(w)$ is the usual notion of length of $w$, denoted $l(w)$. Denote by $w_0$ the maximal length element of $W$.
 At the other extreme 
 \[
l_\Delta(w)
 \]
  is called the \emph{reflection length} of $w$ (and also absolute length or rank). We will use the first terminology. We will need the following property:

  \begin{lemma}[Carter, \cite{Carter}] \label{lem : refl length}
  For each $w \in W$, one has 
  \[
  l_{\Delta}(w) = n - \dim \Ker(w-1).
  \]
  Moreover, $w$ is an involution if and only if $w$ can be written as a product 
  \[
  w = s_{\beta_1} \ldots s_{\beta_r}
  \]
  of reflections associated to pairwise orthogonal roots, and in that case we have $r=l_{\Delta}(w)$. 
  \end{lemma}

For a Weyl group element $w \in W$, denote the subset of positive roots
$$ \Delta_{+}^w := \Delta_+ \cap w (- \Delta_+) =
\{ \beta \in \Delta_+ \mid w^{-1}(\beta) \in - \Delta_+ \}.
$$
  Recall that 
  \begin{equation}
  \label{eq:w-rw}
  \mathbf{w} = (i_1, \ldots , i_l) \in [1, n]^l 
  \end{equation}
  is called a \emph{reduced word} of $w$ if $w = s_{i_1} \ldots s_{i_l}$ and $l=l(w)$. Given such a reduced word, denote
 \[
 w_{\leq k} := s_{i_1} \ldots s_{i_k}
 \]
for $0 \leq k \leq l$. The set $\Delta_{+}^w$ is given by 
\begin{equation}
\label{eq:beta-k}
\Delta_{+}^w = \{ \beta_1, \ldots , \beta_l \}, \quad \text{where} \quad \beta_k := w_{\leq k-1} (\alpha_{i_k}), \; \; \forall 1 \leq k \leq l.
\end{equation}
The set $\Delta_{+}^w$ is convex, and moreover 
\[
\beta_1 \prec_w \cdots \prec_w \beta_l
\]
defines a \emph{convex ordering} on $\Delta_{+}^w$, which means that if $\beta_k + \beta_l \in \Delta_+$ with $k<l$, then there is $k<m<l$ such that $\beta_k + \beta_l = \beta_m$. Note that two distinct reduced expressions of $w$ yield different convex orderings in general. 

 A subword of $\mathbf{w}$ is a word of the form 
 \begin{equation}
 \label{eq:v-sub-w}
 \mathbf{v} := (i_1, \ldots , \widehat{i}_{p_1} , \ldots , \widehat{i}_{p_d} , \ldots , i_l)
 \end{equation}
 obtained from $\mathbf{w}$ by removing its entries in positions $1 \leq p_1<  \cdots < p_d \leq l$ for some $0 \leq d \leq l$. The value of $\mathbf{v}$ is defined to be 
 \[
 s_{i_1} \ldots \widehat{s}_{i_{p_1}} \ldots \widehat{s}_{i_{p_d}} \ldots s_{i_l}.
 \]
 Note that the word $\mathbf{v}$ might not be reduced in general. We have
\begin{align*}
    w &= s_{i_{1}} \dots s_{i_{p_{d}}} \dots s_{i_{l}} = s_{\beta_{p_{d}}} s_{i_{1}}\dots s_{i_{p_{d-1}}} \dots \hat{s}_{i_{p_{d}}} \dots s_{i_{l}} \\
    &= s_{\beta_{p_d}} s_{\beta_{p_{d-1}}} s_{i_{1}}\dots \widehat{s}_{i_{p_{d-1}}} \dots \hat{s}_{i_{p_d}} \dots s_{i_l} = \dots =  s_{\beta_{p_d}} \dots  s_{\beta_{p_1}} v,
\end{align*}
which shows the following:
\ble{l:Bruhat}
Consider a reduced word \eqref{eq:w-rw} of $w\in W$ and a subword \eqref{eq:v-sub-w} of it with value $v \in W$. Then 
\[
    v = s_{\beta_{p_{d}}} \dots s_{\beta_{p_{1}}} w
\]
in terms of the roots \eqref{eq:beta-k}.
\ele

Finally, we recall that the Bruhat order on $W$ is given by setting $v \leq w$ if there is a subword of (one and thus any) reduced word of $w$ with value $v$. The corresponding covering relation will be denoted by 
$v \lessdot w$ $\Leftrightarrow$ $v \leq w$ and $l(v) = l(w)-1$.
   \subsection{The full flag variety $G/B_{+}$ and Richardson varieties }
    \label{subsection: flags and Richardson}
The full flag variety $G/B_+$ of $G$ has the stratifications
\[
G/B_{+} = \bigsqcup_{v \in W} B_{-} v B_{+}/B_{+} =
\bigsqcup_{w \in W} B_{+} w B_{+}/B_{+}
\]
with strata the Schubert cells, $B_{-}vB_{+}/B_{+}$ and $B_{+} w B_{+}/B_{+}$, $v, w \in W$. Their closures in $G/B_+$ are the Schubert varieties
\[
\overline{B_{-}v B_{+}/B_{+}} = 
\bigsqcup_{v' \geq w} B_{-} v' B_{+}/B_{+}, 
\quad 
\overline{B_{+} w B_{+}/B_{+}} = 
\bigsqcup_{w' \leq w} B_{+} w' B_{+}/B_{+}.
\]
The \emph{open Richardson varieties} are the intersections 
\[
\oRvw := B_{-}vB_{+}/B_{+} \cap B_{+}wB_{+}/B_{+}
\]
for $v, w \in W$; $\oRvw$ is non-empty if and only if $v \leq w$ in which case it is an irreducible affine subvariety of $G/B_+$ of dimension $l(w)-l(v)$. They give rise to the stratification 
\begin{equation}
    \label{strat}
G/B_{+} = \bigsqcup_{v \leq w} \oRvw.
\end{equation}
The \emph{closed Richardson variety} $\Rvw$ is the Zariski closure of $\oRvw$ in $G/B_+$; it has the stratification 
\begin{equation}
\label{eq:clos-R}
 \Rvw = \bigsqcup_{v \leq v' \leq w' \leq w} \mathcal{R}_{v'}^{w'}.
\end{equation}
It was proved by Deodhar \cite{Deodhar} that the open Richardson variety $\oRvw$ admits a decomposition indexed by the subset $\mathcal{D}_v^{\mathbf{w}}$ of \emph{distinguished subwords} of $w$ with value $v$. 
Given a reduced word $\mathbf{w}$ for $w$ as in \eqref{eq:w-rw} and a subword $\mathbf{v}$ for $v$ as in \eqref{eq:v-sub-w}, 
we let 
$$ \sigma_j := s_{\beta_{p_1}} \ldots s_{\beta_{p_{l_j}}} w_{\leq j}, \quad   \text{where} \quad l_j := \max \{ 1 \leq k \leq d \mid p_k \leq j \}  $$
for every $1 \leq j \leq l$. Then  $\mathbf{v}$ is called a {\em{distinguished subword}} of $w$ with value $v$ if one has $\sigma_j \leq \sigma_{j-1} s_{i_j}$ for every $1 \leq j \leq l$. It was shown by Deodhar (see {{\cite[Lemma 5.1]{Deodhar}}}) that, if $\mathbf{v}$ is a distinguished subword of $w$ with value $v$, then the quantities $n(\mathbf{v})$ and $m(\mathbf{v})$ defined by  
$$ n(\mathbf{v}) := \sharp \{ 1 \leq j \leq l \mid \sigma_{j-1} = \sigma_j \} \quad \text{and} \quad m(\mathbf{v}) := \sharp \{ 1 \leq j \leq l \mid \sigma_{j-1} > \sigma_j \} $$
satisfy the following relation:
\begin{equation} \label{eq : distinguished}
 n(\mathbf{v}) + 2 m(\mathbf{v}) = l(w)-l(v) . 
 \end{equation}
By way of definition $n(\mathbf{v}) = d$. 
 
 The Deodhar decomposition of the open Richardson variety $\oRvw$ is written as 
\begin{equation}
\label{Deodhar}
\oRvw = \bigsqcup_{\mathbf{v} \in \mathcal{D}_v^{\mathbf{w}}} (\mathbb{C}^{\times})^{n(\mathbf{v})} \times (\mathbb{C})^{m(\mathbf{v})}.
\end{equation}
As shown in \cite[Lemma 3.5]{MarshRietsch}, there is a unique $\mathbf{v}_+ \in \mathcal{D}_v^{\mathbf{w}}$ (called the positive subword), such that $m(\mathbf{v}_+)=0$, i.e., $n(\mathbf{v}_+)=l(w)-l(v)$. 
This implies that there is exactly one Deodhar open stratum of $\oRvw$; its is a torus of dimension equals the dimension of $\oRvw$. 

 The following lemma will be useful to us:

  \ble{lem : refl length less than n} Assume that $\mathbf{w}$ is a reduced word of $w$ and $v \leq w$. 
  Then for any distinguished subword $\mathbf{v}$ of $\mathbf{w}$ with value $v$, we have
    $$ n(\mathbf{v}) \geq l_{\Delta}(vw^{-1}).$$
  \ele
  \begin{proof} Assume the notation \eqref{eq:w-rw} and \eqref{eq:v-sub-w} for $\mathbf{w}$ and $\mathbf{v}$. By Lemma~\ref{l:Bruhat} 
    $$ v = s_{\beta_{p_1}} \ldots s_{\beta_{p_d}} w,  $$
    so
    $$ l_{\Delta}(vw^{-1})  \leq d = n(\mathbf{v}).  $$
  \end{proof}

\subsection{Partial flag varieties and projected Richardson varieties}
\label{section : recollection on partial flags}

Let $P \supseteq B_+$ be a parabolic subgroup of $G$. Denote by $L$ be the Levi subgroup of $P$ containing $T$ and by $\Delta^L$ the root system of $\Lie L$ though as a
subset of $\Delta$. Set $\Delta_{\pm}^L := \Delta_{\pm} \cap \Delta^L$.
 
Let $W_P \subset W$ be the parabolic Weyl group associated to $P$ and $W^P$ be the collection of unique minimal length representatives of the cosets in $W/W_P$. For $w \in W$ denote by $w^P$ the  unique minimal length representative of the class of $w$ in $W/W_P$, thus
$w = w^P w_P$ for a unique $w_P \in W_P$ and $l(w) = l(w^P) + l(w_P)$. The $P$-Bruhat order of Knutson--Lam--Speyer \cite{KLS}
is the partial ordering $\leq_P$ of $W$ which is the transitive closure of the covering relation $\precdot_P$ given by 
\begin{equation}
    \label{eq:Plessdot}
v \precdot_P w \Leftrightarrow \text{$v \precdot w$  and $v W_P \neq wW_P$.} 
\end{equation}
We have,
\[
w \in W^P, \; \; v \leq w \Rightarrow v \leq_P w, 
\]
\cite[Proposition 2.5]{KLS}.

Denote the natural projection 
\begin{equation}
\label{eq:eta}
\eta: G/B_+ \longrightarrow G/P. 
\end{equation}
For $v \leq_P w$ consider the projected open Richardson variety \cite{KLS-C,KLS}
\begin{equation}
\label{eq:projopenRich}
\Pi_v^w := \eta(\oRvw).
\end{equation}
Lusztig \cite{Lusztig} constructed the stratification of $G/P$ 
\begin{equation}
\label{P-strat}
G/P = \bigsqcup_{ \substack{v \leq w \\ w \in W^P}} \Pi_v^w
\end{equation}
that generalizes the stratification \eqref{strat} of $G/B_+$. Following \cite{KLS}, 
define the equivalence relation on the pairs
$\{(v, w) \in W \times W \mid v \leq_P w\}$ generated by the relation $(v, w) \sim (v', w')$ if $v = v' z$, $w = w' z$ for some $z \in W_P$ such that 
$l(v) = l(v') + l(z)$, $l(w) = l(w') + l(z)$.
By \cite[Lemma 3.1]{KLS},
\begin{equation}
\label{proj1}
(v, w) \sim (v', w') \quad \Rightarrow \quad \Pi_{v}^w = \Pi_{v'}^{w'}.
\end{equation}
Furthermore, by \cite[Lemma 2.4]{KLS}, if $v \leq_P w = w^P w_P$ with 
$w^P \in W^P$ and $w_P \in W_P$, then 
\begin{equation}
\label{proj2}
(v,w) \sim (v w_P^{-1}, w^P), \quad
\mbox{and so} \quad
\Pi_v^w = \Pi_{v w_P^{-1}}^{w^P}.
\end{equation}
Combining \eqref{P-strat}--\eqref{proj2}, gives that for all $v, w, v', w'$ with $v \leq_P w$, $v' \leq_P w'$, 
\begin{equation}
\label{proj3}
(v, w) \sim (v', w') \quad \Leftrightarrow \quad \Pi_{v}^w = \Pi_{v'}^{w'}.
\end{equation}
The Zariski closures of $\Pi_v^w$ were first determined in \cite{GY,MarshRietsch}. We will use the following description given in \cite[Proposition 3.6]{KLS}
 \begin{equation} \label{eq : stratif of proj Richardsons}
  \overline{\Pi_v^w} = \eta(\Rvw) = \bigsqcup_{v \leq v' \leq_P w' \leq w} \Pi_{v'}^{w'}, \quad \forall v \leq_P w. 
  \end{equation}

\sectionnew{A stratification of the Poisson degeneracy loci of flag varieties}
In this section we describe the stratification of the Poisson degeneracy locus of every flag variety $G/B_+$ (with respect to the standard Poisson structure) into orbits under the fixed maximal torus of $G$. This stratification is given in terms of open Richardson varieties of a specific combinatorial type, 
which is investigated in detail. 
\subsection{The standard Poisson structure on $G/B_+$}
\label{sec:Poisson-flag}
For each positive root $\beta \in \Delta_+$ fix root vectors $e_\beta$ and $f_\beta$, corresponding to roots 
$\beta$ and $-\beta$ and normalized by
\[
(e_\beta, f_\beta) =1.  
\]
The standard Poisson structure on the flag variety $G/B_+$ is given by the Poisson bivector field 
\begin{equation}
\label{eq:pi}
    \pi = \sum_{\beta \in \Delta_{+}} \chi(e_\beta) \wedge \chi(f_\beta)
\end{equation}
where $\chi: \mathfrak{g} \to \mathrm{Vect}(G/B_{+})$ is the infinitesimal action of $G$ on $G/B^{+}$ (see e.g. \cite{BGY,GY}). Furthermore, the action of the maximal torus $T$ on ($G/B_{+},\pi$) is Poisson, i.e., it preserves the Poisson structure. Alternatively, one can define the Poisson structure $\pi$ to be the push-forward of the standard Poisson structure \cite[Sect. 4.4]{ES} on $G$ under the projection map $G\to G/B_+$. We will need the following results about the Poisson manifold $(G/B_+, \pi)$.
\bth{t:sympl-leaves} For all connected simply connected complex semisimple Lie groups $G$ the following hold:
\begin{enumerate} 
\item[(i)] \cite[Theorem 0.4(i)]{GY} The $T$-orbits of symplectic leaves of $(G/B_{+},\pi)$ are precisely the open Richardson varieties $\oRvw$ for $(v,w)\in W\times W$, $v\leq w$. In particular, all open Richardson varieties are regular Poisson submanifolds of $(G/B_{+},\pi)$.
\item[(ii)] \cite[Theorem 3.1(1)]{Y} The codimension of a symplectic leaf in $\oRvw$, i.e., the corank of $\pi$ in $\oRvw$, equals
\begin{equation*}
    \dim \Ker(w^{-1}v +1).
\end{equation*}
\end{enumerate}
\eth
In connection to the corank property in the second part of the theorem, we have the following:
\ble{l:ker}
For all Weyl group elements $v, w \in W$, 
\begin{multline*}
\dim \Ker(w^{-1}v +1)= 
\dim \Ker(vw^{-1} +1)
\\
= \dim \Ker(w v^{-1} +1) = 
\dim \Ker(v^{-1} w +1).
\end{multline*}
\ele
\begin{proof} The first identity holds because the products in question are conjugated
\[
(v w^{-1} + 1) = w ( w^{-1} v + 1) w^{-1}
\] 
and the third identity is analogous. The second identity follows from the fact that the $-1$ eigenvectors of an operator $A \in GL(V)$ and 
its inverse $A^{-1}$ coincide.
\end{proof}


\bre{symmetries} There are three different types of symmetries of the Poisson degeneracy loci $D_{2k}(G/B_+, \pi)$:
\begin{enumerate}
\item[(1)] The left action of $T$ on $G/B_+$ preserves the Poisson structure $\pi$. This induces an action of $T$ on $D_{2k}(G/B_+, \pi)$.
\item[(2)] Denote by $\Aut(\Gamma)$ the automorphism 
group of the Dynkin graph $\Gamma$ of $G$. Each element $\tau \in \Aut(\Gamma)$ lifts to an automorphism of $\g$ which preserve the bilinear form $(.,.)$ and is defined by 
\[
e_{\alpha_i} \mt e_{\tau(\alpha_i)}, \;\; 
f_{\alpha_i} \mt f_{\tau(\alpha_i)}, \;\; 
h_{\alpha_i} \mt h_{\tau(\alpha_i)}, \;\;
\forall 1 \leq i \leq n.
\]
The latter in turn lifts to an automorphism 
$\Xi_\tau \in \Aut(G)$ such that $\Xi_\tau(B_+) = B_+$.
This gives an action of $\Aut(\Gamma)$ 
on $G/B_+$ that preserves the Poisson structure $\pi$, and thus defines an action of $\Aut(\Gamma)$ on $D_{2k}(G/B_+, \pi)$.
\item[(3)] It is easy to show that there exists a representative $\dot{w}_0 \in N(T)$ of the longest element $w_0$ of $W$ such that 
\begin{equation}
\label{eq:Xi-dot}
\dot{w}_0^2 =1 \quad \mbox{and} \quad \Xi_\tau(\dot{w}_0) = \dot{w}_0, \; \; \forall \tau \in \Aut(\Gamma).
\end{equation}
In the setting of \eqref{eq:pi}, 
\[
\Ad_{\dot{w}_0}(e_\beta) \wedge 
\Ad_{\dot{w}_0}(f_\beta) = - e_\beta \wedge f_\beta,
\quad \forall \beta \in \Delta_+.
\]
Therefore the left multiplication action of $\dot{w}_0$ on $(G/B_+, \pi)$ is anti-Poisson. This gives an action of $\Zset_2$ on on $D_{2k}(G/B_+, \pi)$.
\end{enumerate}
It follows from \eqref{eq:Xi-dot} that the $\Aut(\Gamma)$ and $\Zset_2$-actions on $G/B_+$ 
commute; it is obvious that they normalize the $T$-action. Thus we have an action of 
\[
(\Aut(\Gamma) \times \Zset_2) \ltimes T
\]
on $(G/B_+, \pi)$ (where each elements acts by a Poisson automorphism or anti-automorphism) and on the corresponding Poisson degeneracy loci $D_{2k}(G/B_+, \pi)$. 
\ere
\subsection{Product decompositions and the reflective length}
We will need a couple of intermediate results to describe the strata of the reduced Poisson degeneracy locus $D_0(G/B_+,\pi)_{\redd}$. 

\begin{lemma} \label{l:orth}
    Let $\gamma_{1},\dots, \gamma_{d} \in \Delta$. Then, $\dim \Ker(s_{\gamma_{1}}\dots s_{\gamma_{d}} +1 ) = d$ if and only if $\gamma_{j} \bot   \gamma_{k}, \, \forall j\neq k$.
\end{lemma}

\begin{proof} 
    Denote $V := \text{span}\{\gamma_{1},\dots,\gamma_{d}\}$ and $u := s_{\gamma_1} \ldots s_{\gamma_d}$.
Since $s_{\gamma_{j}}|_{V^{\bot}} = \id|_{V^{\bot}}$ for all $1 \leq j \leq d$,  
\begin{equation}
    \label{eq:2}
    (u+1)|_{V^{\bot}} = 2 \id_{V^{\bot}}.
\end{equation}
Since $V$ is stable under $u$ and the pairing $(\cdot, \cdot)$ is non-degenerate, 
\begin{equation} \label{inclusion}
\Ker(u+1) \subseteq V.
\end{equation} 

Assume now that $\dim \Ker(u+1) = d$. Thus~\eqref{inclusion} implies that $V = \Ker(u+1)$ and hence 
\[
    s_{\gamma_{1}}\dots s_{\gamma_{d}}(\gamma_{j})  = -\gamma_{j},\quad \forall 1\leq j \leq d.
\]
Then we have: 
\begin{align*}
        s_{\gamma_{1}}\dots s_{\gamma_{d}}(\gamma_{j}) &= s_{\gamma_{2}}\dots s_{\gamma_{d}}(\gamma_{j}) - \frac{ 2 (s_{\gamma_{2}}\dots s_{\gamma_{d}}(\gamma_{j}),\gamma_{1})}{ \| \gamma_1 \|^2} \gamma_{1}\\
        &= \dots  = \gamma_{j} - \sum_{i=1}^{d}\frac{2 (s_{\gamma_{m+1}}\dots s_{\gamma_{d}}(\gamma_{j}),\gamma_{m})}{ \| \gamma_m \|^2} \gamma_{m}.
\end{align*}
The coefficient of $\gamma_{d}$ is $2 (\gamma_j, \gamma_d)/ (\gamma_d, \gamma_d)$, and the linear independence of $\gamma_{1},\dots, \gamma_{d}$ implies $\gamma_{d} \bot \gamma_{j}, \, \forall j\neq d$. 

We can commute $s_{\gamma_{d}}$ to the left
($  s_{\gamma_{1}}\dots s_{\gamma_{d}} =  s_{\gamma_{d}} s_{\gamma_{1}}\dots s_{\gamma_{d-1}}$) and iterate the procedure to show that $\gamma_{j} \bot \gamma_{k}, \, \forall j \neq k$.

Conversely, the assumption $\gamma_j \bot \gamma_k$, $\forall j \neq k$ implies that $\gamma_{1},\ldots, \gamma_{d}$ are linearly independent.
Since 
\[
s_{\gamma_k}(\gamma_j)= 
\begin{cases}
\gamma_j, &\mbox{if} \; \; j \neq k
\\
- \gamma_j, &\mbox{if} \; \; j=k,
\end{cases}
\]
$s_{\gamma_{1}} \dots s_{\gamma_{d}} (\gamma_{j})= - \gamma_j$, $\forall 1 \leq j \leq d$. Hence by~\eqref{inclusion} we get $V = \Ker(s_{\gamma_{1}} \dots s_{\gamma_{d}}+1)$ and thus $\dim \Ker(s_{\gamma_{1}} \dots s_{\gamma_{d}}+1) = d$ as $\gamma_1, \ldots , \gamma_d$ forms a basis of $V$. 

\end{proof}

 \bco{cor : involution orth}
 Let $u$ be an involution in $W$ and let $d := l_{\Delta}(u)$. Then for any collection $\gamma_1, \ldots , \gamma_d \in \Delta$ such that $u=s_{\gamma_1} \ldots s_{\gamma_d}$, we have that $\gamma_j \bot \gamma_k, \forall j \neq k$.
 \eco

\begin{proof}
Let $\gamma_1, \ldots , \gamma_d$ such that $u=s_{\gamma_1} \ldots s_{\gamma_d}$. As $u$ is involutive, we have that
$$ \dim \Ker(u+1) = n - \dim \Ker(u-1) = n-(n-l_{\Delta}(u)) = d $$
using Lemma~\ref{lem : refl length}, and hence Lemma~\ref{l:orth} implies that $\gamma_j \bot \gamma_k,  \forall j \neq k$.
\end{proof}
\subsection{The reduced Poisson degeneracy locus $D_0(G/B_+,\pi)_{\redd}$ and open Richardson varieties}
Let $w\in W$ and $\mathbf{w}$ be a reduced word of $w$. For every $v \leq w$ there exists a {\em{reduced subword}} of $\mathbf{w}$ 
with value $v$; it necessarily has length $l(v)$. This subword is not unique: for example, for $w := s_1 s_2 s_1$ and the reduced word $\mathbf{w}:=(1,2,1)$, 
there exist two reduced subwords with value $v=s_1$. The subword of $\mathbf{w}$ with value $v$ is unique \cite[Lemma 3.5]{MarshRietsch} if one requires the stronger property that the subword be {\em{positive}} in the sense of \cite[Definition 3.4]{MarshRietsch}. 

Theorem \ref{t:sympl-leaves} shows that the 
reduced Poisson degeneracy locus $D_0(G/B_+,\pi)_{\redd}$ is stratified by open Richardson varieties. The next theorem characterizes the open Richardson varieties that lie inside $D_0(G/B_+,\pi)_{\redd}$.    
\bth{t:v-w-equiv}
Let $v, w \in W$, $v \leq w$ and $\pi$ be the standard Poisson structure on $G/B_+$ for a connected simply connected complex semisimple Lie group $G$. Fix a reduced word $(i_1, \ldots, i_{l(w)})$ of $w$ and denote
\[
d := l(w) - l(v).
\]
The followings are equivalent:
\begin{enumerate}
    \item $\pi|_{\Rvw} = 0$;
    \item $\pi|_{\oRvw} = 0$;
    \item $\dim\Ker(vw^{-1}+1) = l(w) - l(v)$;
    \item $(v w^{-1})^2 =1$ and $l_\Delta(v w^{-1}) = l(w) - l(v)$;
    \item $v = s_{\gamma_1} \dots s_{\gamma_d} w$ for some $\gamma_1, \ldots, \gamma_d \in \Delta_+$ such that $\gamma_j \bot \gamma_k, \forall j \neq k$;
    \item there exits a reduced subword $(i_1, \ldots , \widehat{i}_{p_1} , \ldots , \widehat{i}_{p_d} , \ldots , i_l)$ of $\mathbf{w}$ with value $v$ 
    for some $1 \leq p_{1} < \dots < p_{d} \leq l$ such that $\beta_{p_{j}} \bot \beta_{p_{j}}$, $\forall j \neq k$, recall the notation \eqref{eq:beta-k}; by Lemma \ref{l:Bruhat},
    $v = s_{\beta_{p_{1}}} \dots s_{\beta_{p_{d}}} w$.
    \end{enumerate}
\eth
\begin{proof} (1)$\Leftrightarrow$(2) because the Poisson structure $\pi$ is algebraic. 

(2)$\Leftrightarrow$(3) follows by combining Theorem \ref{t:sympl-leaves}(ii) and Lemma \ref{l:ker}.

(3)$\Rightarrow$(6) There exits a reduced subword $(i_1, \ldots , \widehat{i}_{p_1} , \ldots , \widehat{i}_{p_d} , \ldots , i_{l(w)})$ 
of the given reduced word of $w$ with value $v$, where $1 \leq p_{1} < \dots < p_{d} \leq l(w)$. 
It follows from Lemma \ref{l:Bruhat} that 
$v = s_{\beta_{p_{1}}} \dots s_{\beta_{p_{d}}} w$.
Condition (3) implies that 
\[
\dim \Ker (s_{\beta_{p_1}} \ldots s_{\beta_{p_d}} +1) = 
\dim \Ker (v w^{-1} +1 ) = l(w) -l(v) = d.
\]
By Lemma \ref{l:orth},
$\beta_{p_{k}} \bot \beta_{p_{m}}, \, \forall k \neq m$.

(6) $\Rightarrow$ (5) is obvious in view of the identity $v = s_{\beta_{p_{1}}} \dots s_{\beta_{p_{d}}} w$ from Lemma \ref{l:Bruhat}.

(5) $\Rightarrow$ (4) $v w^{-1} = s_{\gamma_1} \dots s_{\gamma_d}$. Since 
$\gamma_j \bot \gamma_k$, we have $s_{\gamma_j} s_{\gamma_k} = s_{\gamma_k} s_{\gamma_j}$
for all $j \neq k$. By Lemma \ref{lem : refl length}, $l_\Delta(v w^{-1})= d = l(w) - l(v)$. 

(4)$\Rightarrow$(3) Since $(v w^{-1})^2 =1$, 
\[
\dim\Ker(vw^{-1}+1) = n - \dim\Ker(vw^{-1} -1) = l_\Delta( vw^{-1}) = l(w) - l(v) 
\]
where in the second equation we used Lemma \ref{lem : refl length}. 
\end{proof}
\subsection{The index set $\GCR(W) \subset W \times W$}
\bde{d:GCR}
Denote by 
\[
\GCR(W)
\]
the collection of pairs $(v,w) \in W \times W$ that satisfy the equivalent conditions in the Theorem \ref{t:v-w-equiv}.
\ede

Obviously, the diagonal of $W \times W$ sits inside $\GCR(W)$: 
\begin{equation}
\label{diag-W}
\{ (w,w) \mid w \in W \} \subseteq \GCR(W). 
\end{equation}
The next lemma shows that this is the case for all pairs in the covering relation of the Bruhat order.
\ble{l:diff-one} If $v, w \in W$ and $v \lessdot w$, i.e., $v \leq w$ and $l(v) = l(w) -1$, then 
\[
(v,w) \in \GCR(W). 
\] 
\ele
\begin{proof}
By Lemma \ref{l:Bruhat}, 
\[
v = s_\beta w
\]
for some $\beta \in \Delta_+^w$. Both conditions (3) and (4) in Theorem \ref{t:v-w-equiv} are obviously satisfied, so 
$(v, w) \in \GCR(W)$.  
\end{proof}
\bre{r:gen-cov}
The notation $\GCR(W)$ emphasizes that this set is an {\em{extended covering relation}} for the Bruhat order 
on $W$. 
\ere

The Poisson geometric interpretation of $\GCR(W)$ gives the following two properties:
\ble{l:leq-GCR}
    If $(v, w) \in \GCR(W)$ and $v\leq v'\leq w'\leq w \in W$ then $(v', w') \in \GCR(W)$.
\ele
\begin{proof} By Theorem, \ref{t:v-w-equiv}
\[
\pi|_{\Rvw}. 
\]
Eq. \eqref{eq:clos-R} implies that 
\[
\pi|_{\mathcal{R}_{v'}^{w'}} =0,
\]
and applying again Theorem \ref{t:v-w-equiv} gives $(v', w') \in \GCR(W)$.    
\end{proof}
\subsection{$T$-orbit stratifications of reduced Poisson degeneracy loci}
The action of the maximal torus $T$ on the flag variety $(G/B_+, \pi)$ is Poisson. Therefore the reduced Poisson degeneracy locus $D_0(G/B_+,\pi)_{\redd}$ is stable under the action of $T$. The next result describes its stratification into $T$-orbits.   
\bco{c:T-orbit-stratification}
\hfill 
\begin{enumerate}
\item[(i)] For all $(v,w) \in \GCR(W)$, the action of $T$ on $\oRvw$ is transitive.
\item[(ii)] The stratification of the reduced Poisson degeneracy locus of $D_0(G/B_+,\pi)_{\redd}$ into $T$-orbits is given by
\[
D_0(G/B_+,\pi)_{\redd} = \bigsqcup_{(v,w) \in \GCR(W)} \oRvw. 
\]
\item[(iii)] The reduced Poisson degeneracy locus of $D_0(G/B_+,\pi)_{\redd}$ is connected. 
\end{enumerate}
\eco
\begin{proof} (i) By Theorem \ref{t:sympl-leaves}(i), for all $v, w \in W$, $v \leq w$, $\oRvw$ is a single $T$-orbit of symplectic leaves and by Theorem \ref{t:sympl-leaves}(ii), for $(v,w) \in \GCR(W)$, the symplectic leaves of $\oRvw$ are points. Therefore the $T$-action on $\oRvw$ is transitive for $(v,w) \in \GCR(W)$. 

(ii) This part follows from (i) and Theorems \ref{t:sympl-leaves} and \ref{t:v-w-equiv}.

(iii) This follows from Lemma \ref{l:diff-one}.
\end{proof}
\bre{r:T-act}
It is not true that the action of $T$ on $\oRvw$ is transitive (or equivalently that $\Rvw$ is a toric variety) only if $(v,w) \in \GCR(W)$.
Consider the case when $v=1$ and $w$ equals a Coxeter element 
\[
w = s_{i_1} \ldots s_{i_n},
\]
$i_j \neq i_k$ for $j \neq k$. In this case the roots
$\{ \beta_1, \ldots, \beta_n\}$ given by \eqref{eq:beta-k} form a basis of $Q$ (see for instance  {{\cite[Chap. VI, §1, Proposition 33]{Bourbaki}}}). 
Every element of $\mathcal{R}_1^w$ can be uniquely written in the form 
$$u_{\beta_1}(z_1) \ldots u_{\beta_n}(z_n)wB_+/B_+$$
with $z_1, \dots , z_n \in \mathbb{C}^\times$. The action of $T$ on $\mathcal{R}_1^w$ is transitive because the homomorphism $T \to (\Cset^\times)^n$, given by  
$$ t \in T \mapsto \left( \beta_1(t) , \ldots , \beta_n(t) \right)$$
is surjective by the above stated property of $\{\beta_1, \ldots, \beta_n\}$.
\ere
 
\sectionnew{The open and closed Richardson varieties
in the reduced Poisson degeneracy locus of a flag variety}
In this section we prove that the closed 
Richardson varieties that belong to the reduced Poisson degeneracy locus $D_0(G/B_+,\pi)_{\redd}$ of any flag variety $G/B_+$ are isomorphic to $(\cp)^d$ for some $d \in \Nset$. In other words, this shows that all 
$T$-orbit closures in $D_0(G/B_+,\pi)_{\redd}$ are isomorphic to $(\cp)^d$, 
recall Corollary \ref{c:T-orbit-stratification}. Simultaneously, we prove that the Bruhat intervals $[v,w]$ for all Richardson varieties $\Rvw$ in $D_0(G/B_+, \pi)$ are isomorphic to power sets.

Furthermore, we describe the irreducible components of $D_0(G/B_+,\pi)_{\redd}$, which are shown to be isomorphic to $(\cp)^d$ for positive integers $d$ which are not necessarily the same, i.e., $D_0(G/B_+,\pi)_{\redd}$ is not equidimensional in general. This is illustrated for the cases of the full flag varieties of $SL_3(\Cset)$ and $SL_4(\Cset)$. It is shown 
that even in those simplest cases the Poisson degeneracy locus $D_0(G/B_+,\pi)$ is not reduced.
\subsection{Notation}
\label{sec:nota}
Throughout the section we fix a pair of Weyl group elements 
\[
(v, w) \in \GCR(W)
\]
and set 
\[
l:= l(w), \quad d:= l(w) - l(v), \quad
[1,d] := \{1, \ldots, d \}. 
\]
Fix also fix a reduced word $\mathbf{w}:=(i_1, \ldots , i_l)$ of $w$. By condition (6) in Theorem \ref{t:v-w-equiv}, there exits a reduced subword 
\begin{equation}
    \label{eq:v-word}
\mathbf{v} :=
(i_1, \ldots , \widehat{i}_{p_1} , \ldots , \widehat{i}_{p_d} , \ldots , i_l) 
\end{equation}
of $w$ with value $v$ for some $1 \leq p_{1} < \dots < p_{d} \leq l$ 
such that 
\begin{equation}
\label{eq:beta-orth}
\beta_{p_{j}} \bot \beta_{p_{k}}, \quad \forall j \neq k,
\end{equation}
recall the notation \eqref{eq:beta-k}. By Lemma \ref{l:Bruhat}, 
\[
v = s_{i_1} \ldots \widehat{s}_{i_{p_1}}  \ldots \widehat{s}_{i_{p_d}} \ldots s_{i_l} =  s_{\beta_{p_{1}}} \dots s_{\beta_{p_{d}}} w. 
\]
Define recursively 
\[
v_{\leq k} = 
\begin{cases}
v_{\leq k-1} s_{i_k}, &\mbox{if $i_k$ appears in $\mathbf{v}$} 
\\
v_{\leq k-1}, &\mbox{otherwise}.
\end{cases} 
\]
Since $\mathbf{v}$ is a reduced word with value $v$, 
\begin{equation}
\label{eq:Delta-v}
\Delta_+^v = \{ v_{\leq m-1} (\alpha_{i_m}) \mid 1 \leq m \leq l, m \neq p_1, \ldots, p_d \}.
\end{equation}

For a subset 
\[
K := \{k_1, \ldots k_t \}\subseteq [1,d]
\]
denote the subword of $\mathbf{w}$
\begin{equation}
\label{eq:bfw-k}
\mathbf{w}_K := (i_1, \ldots, \widehat{i}_{p_{k_1}}, \ldots \widehat{i}_{p_{k_t}}, \ldots, i_l)
\end{equation}
with value
\begin{equation}
\label{eq:w_K}
w_K:= s_{i_1} \ldots \widehat{s}_{p_{k_1}}, \ldots 
\widehat{s}_{p_{k_t}} \ldots s_{i_l} = 
\left( \textstyle \prod_{j \in K} s_{\beta_{p_j}}
\right) w. 
\end{equation}
Denote 
\begin{equation}
\label{eq:v_K}
v_K := \left( \textstyle \prod_{j \in K} s_{\beta_{p_j}}
\right) v = w_{[1,d] \backslash K}.
\end{equation}
Clearly
\[
\mathbf{v} = \mathbf{w}_{[1,d]}, \quad
v = w_{[1,d]} 
\]
and
\[
\mathbf{w} = \mathbf{w}_{\varnothing}, \quad
w = v_{[1,d]}.
\]

For all $K \subseteq [1,d]$, 
\[
\Ker (w_Kw^{-1} + 1) = \Span \{ \beta_{p_j} \mid j \in K \}.  
\]
Since $\{\beta_{p_j} \mid j \in [1,d]\}$ are linearly independent we obtain the following:
\ble{l:distinct-w_K}
For all $J, K \subseteq [1,d]$, $J \neq K$, 
\[
w_J \neq w_K.
\]
\ele
\subsection{Properties of the pairwise orthogonal roots associated to $(v,w) \in \GCR(W)$}
\bpr{lemma : positive roots and orderings} For each pair $(v, w) \in \GCR(W)$ as in condition (6) of Theorem \ref{t:v-w-equiv}, the following properties hold:
 \begin{enumerate}
     \item[(i)] For every $1 \leq k \leq d$, one has $v^{-1}(\beta_{p_k}) \in \Delta_{+}$.
     \item[(ii)] For all $1 \leq j < k \leq d$, one has $\beta_{p_k} - \beta_{p_j} \notin - \Delta_+$.
     \item[(iii)] For all $1 \leq j < k \leq d$, one has $v^{-1}(\beta_{p_k} - \beta_{p_j}) \notin \Delta_+$.
\end{enumerate}
    \epr

 \begin{proof}
For $1 \leq k \leq d$ we have 
 \begin{align*}
     v^{-1}(\beta_{p_k}) = w^{-1} \left( \textstyle \prod_{j=1}^d  s_{\beta_{p_j}}\right)(\beta_{p_k}) = - w^{-1}(\beta_{p_k})
 \end{align*}
  by the pairwise orthogonality of $\beta_{p_1}, \ldots , \beta_{p_d}$. As $\beta_{p_k} \in \Delta_+^w$, we have that $w^{-1}(\beta_{p_k}) \in -\Delta_+$, and thus $v^{-1}(\beta_{p_k}) \in \Delta_+$ 
  which proves (i). 

   Fix now $j,k$ with $1 \leq j < k \leq d$. For a proof of (ii) and (iii) by contradiction, assume that $\beta_{p_k} - \beta_{p_j} \in \Delta$. Then
  $$ \beta_{p_j} + \beta_{p_k} =  s_{\beta_{p_j}}(\beta_{p_k}- \beta_{p_j}) \in \Delta $$
    and hence by convexity of $\Delta_+^w$, we have that 
   $$ \beta_{p_j} + \beta_{p_k} \in \Delta_+^w  \quad \text{and} \quad \beta_{p_j} \prec_w \beta_{p_j} + \beta_{p_k} \prec_w \beta_{p_k}.$$
   Therefore there exists an index $t$ such that $p_j < t < p_k$ and 
    \begin{equation} \label{eqn for betaj + betak}
    \beta_{p_j} + \beta_{p_k}= \beta_t= w_{\leq t-1}(\alpha_{i_t})= s_{i_1} \ldots s_{i_t-1}(\alpha_{i_t}) . 
    \end{equation}
   Using Lemma~\ref{l:Bruhat}, we get 
   $$ \beta_{p_k} - \beta_{p_j} = s_{\beta_{p_j}} s_{i_{p_1}} \ldots s_{i_t-1}(\alpha_{i_t}) =  s_{i_1} \cdots \widehat{s}_{i_{p_j}} \cdots s_{i_{t-1}}(\alpha_{i_t}).$$
   Applying again Lemma~\ref{l:Bruhat} we get 
   $$ \beta_{p_k} - \beta_{p_j} =  \left( \textstyle \prod_{m \neq j, p_m <t} s_{\beta_{p_m}} \right) 
   \delta, \quad \text{where} \quad \delta :=v_{\leq i_{t-1}}(\alpha_{i_t}). $$
   As $k>s$, the reflections involved in the product are of the form $s_{\beta_{p_m}}$ with $m \neq j,k$. Hence by the pairwise orthogonality of of $\beta_{p_1}, \ldots , \beta_{p_d}$, 
   they leave $\beta_{p_k} - \beta_{p_j}$ invariant, and thus we obtain 
   $$ \beta_{p_k} - \beta_{p_j} = \delta.$$ 
   Now we claim that $\delta \in \Delta_+^v.$
   First, note that 
   \begin{equation}
   \label{eq:t-p1-d}
   t \notin \{p_1, \ldots , p_d \}. 
   \end{equation}
   Indeed, if $t \in \{p_1, \ldots , p_d \}$, then $w_{\leq t-1}(\alpha_{i_t})$ would be one of the roots $\beta_{p_1}, \ldots , \beta_{p_d}$ other than $\beta_{p_j}$ and $\beta_{p_k}$ (as $p_j<t<p_k$). Hence it would be orthogonal to $\beta_{p_j}$ and $\beta_{p_k}$ which is impossible because \eqref{eqn for betaj + betak} would imply that $(\beta_t, \beta_t) = (\beta_t, \beta_{p_k} + \beta_{p_k}) =0$. Combining \eqref{eq:Delta-v} and \eqref{eq:t-p1-d} gives that $\beta_{p_k} - \beta_{p_j} = \delta \in \Delta_+^v$. 
   
In other words, the assumption $\beta_{p_k} - \beta_{p_j} \in \Delta$ implies that 
$\beta_{p_k} - \beta_{p_j} \in \Delta_+^v$, which proves (ii) and (iii). 
 \end{proof}

\subsection{A map from $(\cp)^d$ to $G/B_+$.}
 For all $(v, w) \in \GCR(W)$ as in the previous subsection, define the following map: 
$$
  \begin{array}{cccc}
 \wt{\kappa} : & L_{\beta_{p_1}} \times \cdots \times  L_{\beta_{p_d}} & \longrightarrow & G/B_+, \\
 {}  & (g_1, \ldots , g_d) & \longmapsto & g_1 \ldots g_d v B_+/B_+ .
 \end{array}
 $$
The next lemma introduces a quotient map  $\kappa$ induced by $\wt{\kappa}$ which will eventually provide the desired closed embedding of algebraic varieties.  

   \bpr{p: j well-def}
   For all $(v, w) \in \GCR(W)$
       the map $\wt{\kappa}$ induces a map 
 \begin{equation}
     \label{eq:kappa}
 \kappa : L_{\beta_{p_1}}/B_{\beta_{p_1}} \times \cdots \times L_{\beta_{p_d}}/B_{\beta_{p_d}} \longrightarrow G/B_{+} . 
 \end{equation}
    \epr

    \begin{proof}
 In order to show that $\wt{\kappa}$ into a well-defined map with domain $L_{\beta_{p_1}}/B_{\beta_{p_1}} \times \cdots \times L_{\beta_{p_d}}/B_{\beta_{p_d}}$, we need to show that 
 \[
 g_1b_1 \cdots g_db_d v B_{+}/B_{+} = g_1 \cdots g_d v B_{+}/B_{+}.
 \]
 We write $b_k = u_k t_k$ with $t_k \in T \cap L_{\beta_{p_k}}$ and $u_k \in U_{\beta_{p_k}}$ for each $1 \leq k \leq d$. Then $t_k$ commutes with $U_{ \pm \beta_{p_m}}$ provided $m \neq k$, and hence it remains to prove that 
 \[
 g_1u_1 \ldots g_du_d v B_{+}/B_{+} = g_1 \cdots g_d v B_{+}/B_{+}.
 \]
 We have 
 $$ u_dvB_{+}/B_{+} = v (v^{-1}u_dv)B_{+}/B_{+} = v B_{+}/B_{+} $$
 because $v^{-1}u_dv \in U_{v^{-1}(\beta_{p_d})} \subset U_{+}$ by Proposition~\ref{lemma : positive roots and orderings}(i). We claim that 
 \begin{equation}
\label{eq:ind step} 
 g_1u_1 \ldots g_ku_k g_{k+1} \cdots g_dvB_{+}/B_{+} = g_1u_1 \cdots g_{k-1}u_{k-1} g_k g_{k+1} \cdots g_dvB_{+}/B_{+}.
 \end{equation}
 The statement of the proposition will then follow by induction. Eq. \eqref{eq:ind step} is equivalent to 
\begin{equation}
\label{eq:ind step 2}
\dot{v}^{-1}g'^{-1}u_kg' \dot{v} \in B_{+}
\end{equation}
where $g' := g_{k+1} \cdots g_d$ and $\dot{v}$ is any fixed representative of $v$ in $N_G(T)$.

   Using the commutation rule \cite[Prop. 8.2.3]{Springer}, we have that 
   \begin{equation}
       \label{eq:comm_rel}
   (g')^{-1}u_k g' = u_k u' \quad \text{with} \enspace u' \in \prod_{ \substack{ m_k , m_{k+1}, \ldots , m_d \in \mathbb{Z} \\ m_k  >0 } } U_{m_k \beta_{p_k} + \ldots + m_d \beta_{p_d}} . 
   \end{equation}
   However, 
   \[
   \|m_k \beta_{p_k} + \ldots + m_d \beta_{p_d}\|^2 =
   m_k^2 \| \beta_{p_k} \| + \cdots 
   + m_d^2 \| \beta_{p_d} \|^2. 
   \]
   If $m_k \beta_k + \cdots + m_d \beta_d \in \Delta$, then either
   \begin{enumerate}
\item $m_k=1$ and $m_j = 0$ for all $j \neq k$ or \item $m_k=1$, $m_{k'}=\pm 1$ for some $k < k' \leq d$ and $m_j = 0$ for all $j \neq k, k'$.
   \end{enumerate}
   In the first case, by Proposition~\ref{lemma : positive roots and orderings}(i) we have $v^{-1}(\beta_{p_k}) \in \Delta_{+}$, which proves \eqref{eq:ind step 2}. Consider the second case. If $\beta_{p_k} - \beta_{p_{k'}} \in \Delta$, then 
   $$ v^{-1}(\beta_{p_k} - \beta_{p_{k'}})\in \Delta_{+} $$
   by Lemma~\ref{lemma : positive roots and orderings}(iii). 
   If $\beta_{p_k} + \beta_{p_{k'}} \in \Delta$, then $\beta_{p_k}- \beta_{p_{k'}}= s_{\beta_{p_{k'}}}(\beta_{p_k} + \beta_{p_{k'}}) \in \Delta$ and again by Lemma~\ref{lemma : positive roots and orderings}(i,iii),
   \[
   v^{-1} (\beta_{p_{k'}}), v^{-1}(\beta_{p_k} - \beta_{p_{k'}}) \in \Delta_+.
   \]
   Hence, $v^{-1}(\beta_{p_k} + \beta_{p_{k'}}) \in \Delta_+$ and in all cases $u' \in B_+$, so \eqref{eq:ind step 2} follows from \eqref{eq:comm_rel}.  
    \end{proof}
\subsection{A stratification of $(\cp)^d$ and injectivity of $\kappa$}
\label{section : stratification}
The domain of $\kappa$ admits the following decomposition:
  \begin{equation} \label{eq : stratif of domain of kappa}
 L_{\beta_{p_1}}/B_{\beta_{p_1}} \times \cdots \times L_{\beta_{p_d}}/B_{\beta_{p_d}} = \bigsqcup_{ \substack{ J,K \subset [1,d] \\ J \cap K = \varnothing }} A_{J,K}, 
 \end{equation}
 where for each pair $(J,K)$ of disjoint subsets of $[1,d]$ the subset $A_{J,K}$ of  $L_{\beta_{p_1}} / B_{\beta_{p_1}} \times \cdots \times L_{\beta_{p_d}} / B_{\beta_{p_d}}$  is defined as follows :
  \begin{align} \label{def AJK}
  A_{J,K} := \big\{ &\big(g_1 B_{\beta_{p_1}}/B_{\beta_{p_1}},  \ldots ,  g_d B_{\beta_{p_d}} /B_{\beta_{p_d}} \big)\mid 
  \mbox{ $g_k = 1$ for $k \in K$}, 
  \\
  &\mbox{$g_k = \dot{s}_{\beta_{p_k}}$ for $k \in J$,
  $g_k \in U_{- \beta_{p_k}}^{\times}$ otherwise} \big\}.  
  \nn
  \end{align}
  The next statement allow us to identify this decomposition with the  stratification of the Richardson varieties $\Rvw$ for $(v, w) \in \GCR(W)$ as orbits for the natural $T$-action on $G/B_{+}$, recall Corollary \ref{c:T-orbit-stratification}(ii).
  
  \bpr{prop: inclusion general case}
      For all $(v,w) \in \GCR(W)$ the following hold for the map $\kappa$ defined in \eqref{eq:kappa}: 
      \begin{enumerate}
      \item[(i)] For $J,K \subseteq [1, d]$ with $J \cap K = \varnothing$,
      \[
      \kappa(A_{J,K}) \subset  \mathcal{R}_{v_J}^{w_K},
      \]
      recall the notation \eqref{eq:w_K} and \eqref{eq:v_K}.
      \item[(ii)] The map $\kappa$ is injective.
      \end{enumerate}
      \epr

 \begin{proof} Let $g_k \in L_{\beta_{p_k}}$ be as in \eqref{def AJK} for $1 \leq k \leq d$, and 
\[
 x:= \big(g_1 B_{\beta_{p_1}}/B_{\beta_{p_1}},  \ldots ,  g_d B_{\beta_{p_d}} /B_{\beta_{p_d}} \big).
\] 
Denote
\[
[1,d] \backslash (K \sqcup J) =
\{t_1, \ldots, t_m\}.
\]
Pulling the terms $g_k,$ $k \notin K \sqcup J$
in $g_1 \ldots g_d v B_+/B_+= \kappa(x)$ to the left gives that
\[
\kappa(x) = u_1 \ldots u_m v_J B_+/ B_+
\]
where
\[
u_q = \Big( \prod_{k \in J, k < q} s_{\beta_{j_k}} \Big)
g_q \Big( \prod_{k \in J, k < q} s_{\beta_{j_k}} \Big)^{-1} \in U_{-\beta_{j_q}}.
\]
The inclusion follows from the fact that the roots $\beta_{p_k}$ are pairwise orthogonal for $1 \leq k \leq d$. Since the vectors $\beta_{p_k}$ are linearly independent, this implies that
\begin{equation}
\label{eq:inclu-kappa}
\kappa(A_{J,K}) \subseteq B_- v_J B_+/B_+
\end{equation}
and that the restriction 
\begin{equation}
\label{eq:kappa-restr}
\kappa|_{A_{J,K}} \; \; \mbox{is injective}.
\end{equation}

 Recalling \eqref{eqn: important identity}, the assumption on the elements $g_k \in L_{\beta_{j_k}}$ implies that
\[
 x:= \big(g'_1 B_{\beta_{p_1}}/B_{\beta_{p_1}},  \ldots ,  g'_d B_{\beta_{p_d}} /B_{\beta_{p_d}} \big)
\] 
where $g'_k = g_k =1$ for $k \in K$, $g'_k = g_k = \dot{s}_{\beta_k}$ for $k \in J$, and $g'_k = u'_k s_{\beta_{j_k}}$ for $k \notin K \sqcup J$ with 
$u'_k \in U_{ \beta_{j_k}}$. Pulling the terms $u'_k,$ $k \notin K \sqcup J$
in $g'_1 \ldots g'_d v_J B_+/B_+= \kappa(x)$ to the left gives that
\[
\kappa(x) = u''_1 \ldots u''_m w_K B_+/ B_+,
\]
where
\[
u''_q = \Big( \prod_{k \notin K, k < q} s_{\beta_{j_k}} \Big)
u'_q \Big( \prod_{k \notin K, k < q} s_{\beta_{j_k}} \Big)^{-1} \in U_{\beta_{j_q}}.
\]
(Once again, the inclusion follows from the fact that the roots $\beta_{p_k}$ are pairwise orthogonal for $1 \leq k \leq d$.) Therefore
\begin{equation}
    \label{eq:inclu-kappa2}
\kappa(x) \in B_+ w_K B_+/B_+.
\end{equation}

Part (i) of the proposition follows from Eqs. \eqref{eq:inclu-kappa}
and \eqref{eq:inclu-kappa2}. 

Lemma \ref{l:distinct-w_K} implies that the images of the restrictions 
$\kappa|_{A_{J,K}}$ are disjoint. This fact and Eq. \eqref{eq:kappa-restr}
imply part (ii) of the proposition.
\end{proof}

\subsection{Isomorphism property of $\kappa$}
Next we prove that $\kappa$ is an isomorphism between 
\[
L_{\beta_{p_1}}/B_{\beta_{p_1}} \times \cdots \times L_{\beta_{p_d}}/B_{\beta_{p_d}} \quad \mbox{and} \quad \Rvw
\]
for all $(v, w) \in \GCR(W)$. First, we describe the image of the restrictions of $\kappa$
to the strata $A_{J,K}$ of $(\cp)^d$. Note that, 
Proposition \ref{prop: inclusion general case}(i) and Lemma \ref{l:leq-GCR} imply that for $(v,w) \in \GCR(W)$,
\begin{equation}
\label{eq:vJwK}
(v_J,w_K) \in \GCR(W), \quad \forall  J,K \subseteq [1,d], \; J \cap K = \varnothing. 
\end{equation}

 \bpr{prop : equality}
  For any $(v,w) \in \GCR(W)$ and  for any pair $J,K$ of disjoint subsets of $\{1, \ldots , d \}$, in the notation of~\eqref{def AJK} we have that 
  $$ \kappa(A_{J,K}) = \mathcal{R}_{v_J}^{w_K} . $$
 \epr

 We start with the following auxiliary lemma. 

\ble{l:uniq-sub}
Let $(v, w) \in \GCR(W)$ and let $J,K$ be two disjoint subsets of $\{1, \ldots , d \}$ as in Section \ref{sec:nota}. Let $\mathbf{w}_K$ denote the reduced expression of $w_K$ obtained by deleting the letters $i_{p_k}, k \in K$. Then there is a unique distinguished subword of $\mathbf{w}_K$ with value $v_J$. In other words, there is a unique Deodhar stratum of $\mathcal{R}_{v_J}^{w_K}$ in the decomposition \eqref{Deodhar}. 
\ele
\begin{proof} 
Denote by $\mathbf{v}_J$ the unique positive subword of $\mathbf{w}_K$ with value $v_J$ (recall that the uniqueness is guaranteed by \cite[Lemma 3.5]{MarshRietsch}) and let $\mathbf{v}'$ be any other distinguished subword of $\mathbf{w}_K$ with value $v_J$. Then we have that $m(\mathbf{v}') \neq 0$ so recalling~\eqref{eq : distinguished}  and using Lemma~\ref{lem : refl length less than n} we get
$$  l_{\Delta}(v_J w_K^{-1}) \leq n(\mathbf{v}') = l(w_K)-l(v_J) - 2 m(\mathbf{v}') < l(w_K)-l(v_J).  $$
Eq. \eqref{eq:vJwK} implies that 
$$l_{\Delta}(v_Jw_K^{-1}) = l(v_J)-l(w_K),$$
which is a contradiction.
\end{proof}

Lemma \ref{l:uniq-sub} and \cite[Theorem 1.3]{Deodhar} imply the following:
\bco{R-polynomial} For all $(v,w) \in \GCR(W)$, the associated Kazhdan--Lusztig $R$-polynomial is 
    \[
    R_{v,w}(q) = (q-1)^{l(w)-l(v)}.
    \]
\eco

   \begin{proof}[Proof of Proposition~\ref{prop : equality}]
 Let $x \in \mathcal{R}_{v_J}^{w_K}$. By the previous lemma, $x$ belongs to the open Deodhar stratum associated to $(\mathbf{v}_J, \mathbf{w}_K)$. This stratum is characterised in {{\cite[Propositions 5.2]{MarshRietsch}}}. From this result we obtain that $x$ can be written as 
 $$ x= \dot{s}_{i_1} \ldots \dot{s}_{i_{p_1-1}} u_1 \dot{s}_{i_{p_1+1}} \ldots \dot{s}_{i_{p_d-1}}u_d \dot{s}_{i_{p_d+1}} \ldots \dot{s}_{i_l} B_+ $$
 where for each $1 \leq j \leq d$ we have
 $$ u_j  
  \begin{cases}
      =1, &\text{if $j \in K$} \\
      =\dot{s}_{i_{j}}, & \text{if $j \in J$} \\
      \in  U_{-\alpha_{i_j}}^{\times}, & \text{otherwise.}
\end{cases}
 $$
 Hence, 
 $$ x = \tilde{u}_1 \ldots \tilde{u}_d \dot{s}_{i_1} \ldots \widehat{s}_{i_{p_1}} \ldots \widehat{s}_{i_{p_d}} \ldots \dot{s}_{i_l} B_+/B_+ = \tilde{u}_1 \ldots \tilde{u}_d v B_+/B_+ = \kappa(\tilde{u}_1, \ldots , \tilde{u}_d), $$
 where for every $1 \leq j \leq d$ we have
 $$ \tilde{u}_j = v_{\leq p_j-j} u_j v_{\leq p_j-j}^{-1} $$
 with 
 $$ v_{\leq p_j-j} =s_{i_1} \ldots \widehat{s}_{i_{p_1}} \ldots \widehat{s}_{i_{p_{j-1}}} \ldots s_{i_{p_j-1}} =  s_{\beta_{p_1}} \ldots s_{\beta_{p_{j-1}}} w_{\leq p_j-1} .$$
 Therefore,
 $$ \tilde{u}_j 
 \begin{cases}
     = 1, & \text{if $j \in K$} \\
     = \dot{s}_{s_{\beta_{p_1}} \ldots s_{\beta_{p_{j-1}}}} w_{\leq p_j-1}(\alpha_{i_{p_j}}) = \dot{s}_{s_{\beta_{p_1}} \ldots s_{\beta_{p_{j-1}}}(\beta_{p_j})} = \dot{s}_{\beta_{p_j}}, &\text{if $j \in J$} \\
     \in U_{- s_{\beta_{p_1}} \ldots s_{\beta_{p_{j-1}}} w_{\leq p_j-j}(\alpha_{i_{p_j}})}^{\times} = U_{- s_{\beta_{p_1}} \ldots s_{\beta_{p_{j-1}}}(\beta_{p_j})}^{\times} = U_{- \beta_{p_j}}^{\times}, &\text{otherwise.}
     \end{cases}
     $$
 Comparing with~\eqref{def AJK}, this gives that $(\tilde{u}_1, \ldots, \tilde{u}_d) \in A_{J,K}$, so $x \in \kappa(A_{J,K})$. Thus we have shown that $\mathcal{R}_{v_J}^{w_K} \subseteq \kappa(A_{J,K})$, and combining with Proposition~\ref{prop: inclusion general case}, we obtain the desired equality. 
      \end{proof}

\bth{t:kappa-isom} For all connected simply connected complex semisimple Lie groups $G$ and
$(v, w) \in \GCR(W)$, the map $\kappa$ defines an isomorphism of algebraic varieties 
\begin{equation}
    \label{eq:kappa-isom}
    L_{\beta_{p_1}}/B_{\beta_{p_1}} \times \cdots \times L_{\beta_{p_d}}/B_{\beta_{p_d}} \simeq \Rvw.
\end{equation}
    Under this isomorphism, the orbits for the natural $(\Cset^\times)^d$-action on $(\cp)^d$ match with the open Richardson varieties $\mathcal{R}_{v'}^{w'}$ for $v \leq v' \leq w' \leq w$.
 \eth
 \bre{Ohn} A special case of the theorem was proved by Ohn in \cite{Ohn}. The statement in \cite{Ohn} is  phrased in terms of orthocells in $W$, and in our language, \cite{Ohn} proves that 
 \[
 \overline{\mathcal{R}_{1,w}} \cong (\cp)^{l(w)} \quad \mbox{for} \quad (1,v) \in \GCR(W).
 \]
 An additional left translate of $\overline{\mathcal{R}_{1,w}}$ by a Weyl group  element is considered in \cite{Ohn} but that does not affect the statement since the action of $G$ on $G/B_+$ is by automorphisms. The statement in \cite {Ohn} is simpler to prove than ours since the case of $(v,w) \in \GCR(W)$ when $v =1$ is simply a consideration of families of pairwise orthogonal roots by condition (5) in Theorem \ref{t:v-w-equiv}. \ere

\begin{proof}[Proof of Theorem \ref{t:kappa-isom}]
It follows from Proposition \ref{prop: inclusion general case}(ii) and Proposition \ref{prop : equality} that $\kappa$ is a bijection between the two sides of \eqref{eq:kappa-isom}. We will show that the differential 
\[
d \, \kappa_x : T_p \, \big( L_{\beta_{p_1}} / B_{\beta_{p_1}} \times \cdots \times L_{\beta_{p_d}} / B_{\beta_{p_d}} \big) \to T_{\kappa(x)} \, \big( G/B_+ \big)
\]
is injective everywhere, which would imply that $\kappa$ defines an isomorphism between the two sides of \eqref{eq:kappa-isom}, see e.g. \cite[Corollary 14.10]{Harris}. 

For each subset $K \subseteq [1,d]$, consider the 
product of shifted Schubert cells
$$ V_K := \{ (t_1 u_1 B_{\beta_{p_1}} , \ldots , t_d u_d B_{\beta_{p_d}}) \mid u_k \in U_{- \beta_{p_k}} \} \subset L_{\beta_{p_1}} / B_{\beta_{p_1}} \times \cdots \times L_{\beta_{p_d}} / B_{\beta_{p_d}},  $$
where 
$$ t_k := \begin{cases} 
 \dot{s}_{\beta_{p_k}} & \text{ if $k \in K$} \\
 1 & \text{if $k \notin K$.}
 \end{cases}
 $$
 The subsets $V_K$ form an open affine covering of $L_{\beta_{p_1}} / B_{\beta_{p_1}} \times \cdots \times L_{\beta_{p_d}} / B_{\beta_{p_d}}$. The restriction of $\kappa$ to $V_K$ is given as follows. For $(u_1, \ldots , u_d) \in U_{- \beta_{p_1}} \times \cdots \times U_{- \beta_{p_d}}$,
 \begin{multline*}
     \kappa(t_1 u_1 B_{\beta_{p_1}} , \ldots , t_d u_d B_{\beta_{p_d}}) = t_1 u_1 \ldots t_d u_d v B_{+}/B_+ =\\ 
     = \big(  \textstyle \prod_{k \in K} \dot{s}_{\beta_{p_k}} \big)  \dot{v}  u'_1 \ldots u'_d B_{+}/B_+ \in \big(  \textstyle \prod_{k \in K} s_{\beta_{p_k}} \big)  v  U_- B_+/B_+,
     \end{multline*}
where 
\[
u'_j = \dot{v}^{-1} \big(  \textstyle \prod_{k \in K, k>j} \dot{s}_{\beta_{p_k}} \big)^{-1} u_j \big(  \textstyle \prod_{k \in K, k>j} \dot{s}_{\beta_{p_k}} \big) \dot{v} \in U_{-v^{-1}(\beta_{p_j})}.
\]
Here we use that $\dot{s}_{\beta_{p_k}}$ normalizes $U_{\beta_{p_j}}$ for
$ k \neq j$ by the orthogonality of $\beta_{p_k}$ and $\beta_{p_j}$. We can identify 
\begin{align*}
&V_K \simeq U_{- \beta_{p_1}} \times \cdots \times U_{- \beta_{p_d}}
\\
&\big(  \textstyle \prod_{k \in K} \dot{s}_{\beta_{p_k}} \big)  \dot{v}  U_- B_+/B_+ \simeq U_-,    
\end{align*}
and in this identification, the restriction of $\kappa$ to $V_K$ is given by 
\[
(u_1, \ldots , u_d) \mapsto u'_1 \ldots u'_d.
\]
This map has everywhere injective differential 
because this is true for the product map 
\[
U_{\gamma_1} \times \cdots \times U_{\gamma_d} \to U_+
\]
for each collection of distinct positive roots 
$\gamma_1, \ldots, \gamma_d$, see e.g. 
\cite[Proposition 8.2.1]{Springer}.
The last statement of the theorem follows from 
Proposition \ref{prop : equality}.
\end{proof}
\subsection{The Bruhat interval $[v,w] \in W$} Assume the setting of Section \ref{sec:nota} for $v, w \in W$ and recall the notation $w_K$ from \eqref{eq:w_K} for $K \subseteq [1,d]$. From eq. \eqref{eq:vJwK}, we have 
\[
v \leq w_K \leq w, \quad \forall K \subseteq [1,d],
\]
for instance, $w = w_\varnothing$ and $v = w_{[1,d]}$. This leads to the question of whether all elements in the Bruhat interval $[v,w]$ in $W$ for $(v,w) \in \GCR(W)$ are of this form and what the Bruhat order relations between them are. This can be done using the algebro-geometric results from the previous subsection: 
\bth{t:Bruhat-interval} For all connected simply connected complex semisimple Lie groups $G$ and $(v,w) \in \GCR(W)$, the Bruhat interval $[v,w]$ in $W$ is isomorphic to 
the power set of $[1,d]$ with the reverse order, i.e., 
the poset of subsets of $[1,d]$ with order given by reverse inclusion. 

In other words, all elements $w' \in W$ with $v \leq w' \leq w$ are of the form
\[
w' = w_K \quad \mbox{for some} \quad K \subseteq [1,d]
\]
and
\[
w_K \leq w_J \quad \Leftrightarrow \quad K \supseteq J.
\]
\eth

We note that for a set $S$, the power set of $S$ is isomorphic to the power set of $S$ with the reverse order by considering complements of subsets.   
 
\begin{proof} Combining Proposition \ref{prop: inclusion general case} and Theorem \ref{t:kappa-isom} gives that
\[
\Rvw = \bigsqcup_{J, K \supseteq [1,d], J \cap K = \varnothing} \mathcal{R}_{v_J}^{w_K}.
\]
Comparing this with the description \eqref{eq:clos-R} of the closure of open Richardson varieties gives that the pairs 
\[
(v', W') \in W \times W \quad \mbox{such that} \quad v \leq v' \leq w' \leq w
\]
are precisely the pairs of the form
\[
(v_J, w_K) \quad \mbox{such that} \quad J, K \supseteq [1,d], J \cap K = \varnothing.
\]
Taking into account that $v_J = w_{[1,d] \backslash J}$ leads to the statement of the theorem. 
\end{proof}
\bre{r:sub-word}
Consider a reduced word $(i_1, \ldots i_l)$ and reduced subword of it
\[
(i_1, \ldots , \widehat{i}_{p_1} , \ldots , \widehat{i}_{p_d} , \ldots , i_l).
\]
Theorem \ref{t:Bruhat-interval} implies that if their values $w$ and $v$ satisfy 
any of the equivalent conditions of Theorem \ref{t:v-w-equiv}, 
then all intermediate words (i.e, subwords of the former that include the latter)
are reduced and their values are precisely the Bruhat interval $[v,w]$. 

If we drop the assumptions of Theorem \ref{t:v-w-equiv}, all of these properties fail
even for positive subwords.
For example, consider the reduced word $(1,2,1)$ in $S_3$ and the positive subword 
$(\widehat{1},\widehat{2},1)$. The intermediate subword $(1, \widehat{2},1)$ is not reduced.
\ere

\subsection{ The irreducible components of the reduced Poisson degeneracy locus}
Denote
\[
W_{\leq} := \{ (v, w) \in W \times W \mid v \leq w \} 
\]
and the induced partial order on it
\begin{equation}
\label{eq:ord-GCR2}
(v, w) \leq (v', w') \quad \Leftrightarrow \quad v \leq v' \leq w' \leq w.
\end{equation}
Lemma \ref{l:leq-GCR} and Theorem \ref{t:Bruhat-interval} imply that this partial order has the following properties:
\bco{c:max-CGR}
    The following hold for all $(v,w) \in \GCR(W)$:
    \begin{enumerate}
    \item[(i)] If  $(v', w') \in W_{\leq}$ and $(v',w') \leq (v,w)$, then $(v', w') \in \GCR(W)$.
    \item[(ii)] In the notation of Section \ref{sec:nota}, the elements $(v', w') \in \GCR(W)$ such that $(v', w') \leq (v,w)$ are precisely the elements 
    \[
    (v_J, w_K) \quad \mbox{for} \quad J, K \subseteq [1,d], J \cap K = \varnothing.
    \]
    \end{enumerate}
\eco
Denote by
\[
\max \GCR(W)
\]
the set of maximal elements of $\GCR(W)$ with respect to the partial order \eqref{eq:ord-GCR2}. Recall from Corollary \ref{c:T-orbit-stratification} that the reduced Poisson degeneracy 
locus $D_0(G/B_+,\pi)_{\redd}$ of $G/B_+$ is connected. 
Theorem \ref{t:v-w-equiv} and Eq. \eqref{eq:clos-R} give the following result:
\bpr{p:irr-comp}
 For every connected simply connected complex semisimple Lie group $G$, 
 the irreducible components of the reduced Poisson degeneracy locus $D_0(G/B_+,\pi)_{\redd}$ are
 \[
 \Rvw \simeq (\cp)^{l(w) - l(v)}
 \]
for $(v,w) \in \max \GCR(W)$.
\epr
Next we present two examples of the sets of irreducible components of the reduced Poisson degeneracy loci 
\[
D_0(SL_{n+1}(\Cset)/B_+, \pi)_{\redd}
\]
for $n=2$ and $3$. The second example shows that the reduced Poisson degeneracy locus $D_0(G/B_+,\pi)_{\redd}$ is not equidimensional in general. 

The automorphism group of the Dynkin graph of $\sl_{n+1}$ for $n > 1$ is $\Aut(\Gamma) \cong \Zset_2$ and Remark \ref{symmetries} gives an action of 
\begin{equation}
\label{eq:Z2Z2}
(\Zset_2 \times \Zset_2) \ltimes T \quad 
\mbox{on} \quad D_0(SL_{n+1}(\Cset), B_+, \pi). 
\end{equation}

Figure 1 shows the Bruhat graphs of $S_3$ and $S_4$;
we use the one-line notation for permutations. 
We identify the vertex $w \in S_{n+1}$ of the Bruhat graph with the Richardson variety
\[
\overline{\mathcal{R}_{w,w}} = w B_+/B_+
\]
and the edge between $v \lessdot w$ with the Richardson variety 
\[
\Rvw \cong \cp.
\]
The action of $\Zset_2 \times \Zset_2$ on this collection of Richardson varieties inside $D_0(SL_{n+1}(\Cset), B_+, \pi)$ is given by the reflections with respect to the central horizontal and vertical lines.

\bex{ex:sl3-1}
Consider the case $G = SL_3(\Cset)$. There are no pairs of orthogonal roots and  
\[
\GCR(W) = \{ (v, w) \in W \times W \mid v=w \; \; \mbox{or} \; \; v \lessdot w \}. 
\]
 \begin{figure}
\label{fig1}
    \centering
    \includegraphics[scale=0.5]{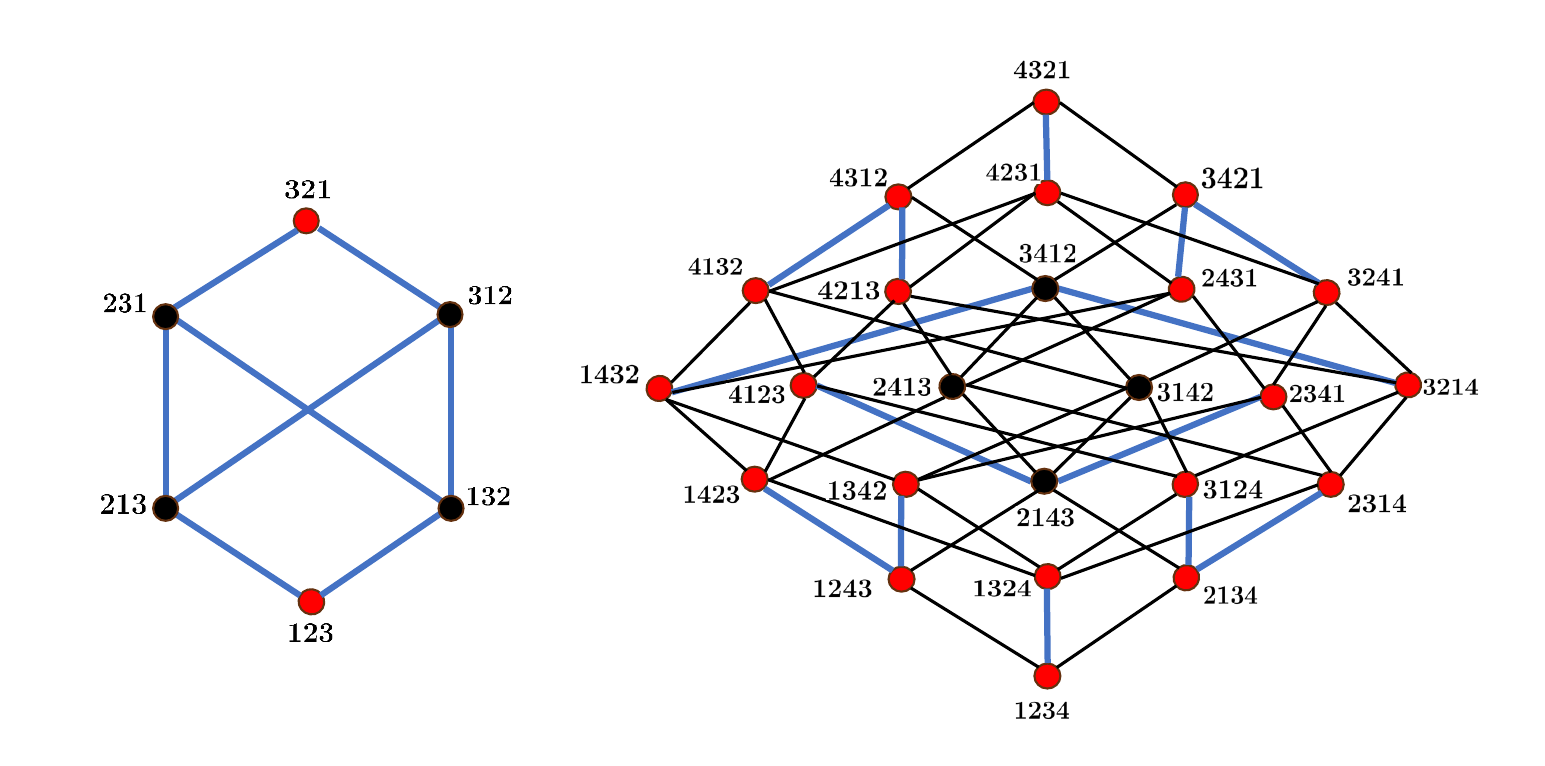}
    \caption{Each edge represents an irreducible Richardson variety which is isomorphic to a $\cp$}
    \label{'Buhat Graph'}
\end{figure}
The reduced Poisson degeneracy locus
$D_0(SL_3(\Cset)/B_+,\pi)_{\redd}$ is given by
the Bruhat graph of $S_2$ on the first picture in Figure 1 with a copy of $\cp$ for each edge and the corresponding projective lines intersecting at the vertices as on the diagram.
\eex
\bex{ex:sl4-1}
Consider the case $G = SL_4(\Cset)$. Now
\[
\GCR(W) \supsetneq \{ (v, w) \in W \times W \mid v=w \; \; \mbox{or} \; \; v \lessdot w \}. 
\]
The reduced Poisson degeneracy locus $D_0(SL_4(\Cset)/B_+,\pi)_{\redd}$ contains a copy of $\cp$ for each edge the Bruhat graph of $S_3$ on the second picture in Figure 1, with the copies of $\cp$ 
intersecting in the same way as the incidence relations for the edges of the Buhat graph. There are a total of 14 pairs $(v,w) \in \GCR(W)$ with $l(w)-l(v) = 1$ that are irreducible components. They are precisely 
the edges of the second picture in Figure 1 colored in blue; the edges in black correspond to closed Richardson varieties embedded in 2-dimensional irreducible components of $D_0(SL_4(\Cset)/B_+,\pi)_{\redd}$.  
%
%
%
%
%
%

Furthermore, $SL_4(\Cset)$ has 3 pairs of positive orthogonal roots:
\begin{equation}
\label{eq:orthog}
(\alpha_{1},\alpha_{3}), \quad (\alpha_{2},\alpha_{1}+\alpha_{2}+\alpha_{3}) \quad \mbox{and} \quad 
(\alpha_{1}+\alpha_{2},\alpha_{2}+\alpha_{3}).
\end{equation}
There are a total of 11 pairs $(v, w) \in \GCR(W)$ with $l(w)- l(v)=2$, which are listed as follows based on the length of the element $v$
\bigskip

\textbf{$\alpha_{1},\alpha_{3}$:}
\begin{tabular}{|c|c|}
    \hline
    $l(v) = 0$:& $(1234,2143)$ \\
    $l(v) = 1$:& $(1324,2413)$\\
    $l(v) = 2$:& $(1342,2431), (3124,4213)$\\
    $l(v) = 3$:& $(3142, 4231)$\\
    $l(v) = 4$:& $(3412,4321)$\\
    \hline
\end{tabular}
\bigskip

\textbf{$\alpha_{1}+\alpha_{2}, \alpha_{2}+\alpha_{3}$:}
\begin{tabular}{|c|c|}
    \hline
   $l(v) = 1$: &$(1324,3142)$\\
   $l(v) = 3$:& $(2413,4231)$\\
    \hline
\end{tabular}
\bigskip

\textbf{$\alpha_{2},\alpha_{1}+\alpha_{2}+\alpha_{3}$:}
\begin{tabular}{|c|c|}
    \hline
     $l(v) = 2$: & $(1423,4132), (2143,3412), (2314,3241)$ \\
    \hline
\end{tabular}
\bigskip

They give rise to 11 irreducible components of $D_0(SL_4(\Cset)/B_+,\pi)_{\redd}$ that are isomorphic to $(\cp)^2$. Figure 2 illustrates 
them: there are 6 components corresponding to the square faces of the polytope colored in blue and 5 components corresponding to the colored planes. 

The intersection of each two irreducible components of $D_0(SL_4(\Cset)/B_+, \pi)_{\redd}$ is either empty or a point. To see this, assume that two irreducible components have a $\cp$ in common; each of them will have to be isomorphic to a $(\cp)^2$. Their intersection will have to be $\overline{\mathcal{R}_{v'}^{w'}}$ for some $v' \lessdot w'$, i.e., $v' = s_\beta w'$, and the two components would have to be of the form 
$\overline{\mathcal{R}_{v'}^w}$ for some $w \in W$ such that $w'\lessdot w$ and $(v', w) \in \GCR(W)$ or $\overline{\mathcal{R}_v^{w'}}$ for some $v \in W$ such that $v \lessdot v'$ and $(v,w') \in \GCR(W)$. Theorem \ref{t:v-w-equiv} would imply that the root $\beta$ belongs to two different pairs of positive orthogonal roots of $SL_4(\Cset)$ which contradicts \eqref{eq:orthog}. 
\begin{figure}
\label{fig2}
    \centering
    \includegraphics[scale = 0.20]{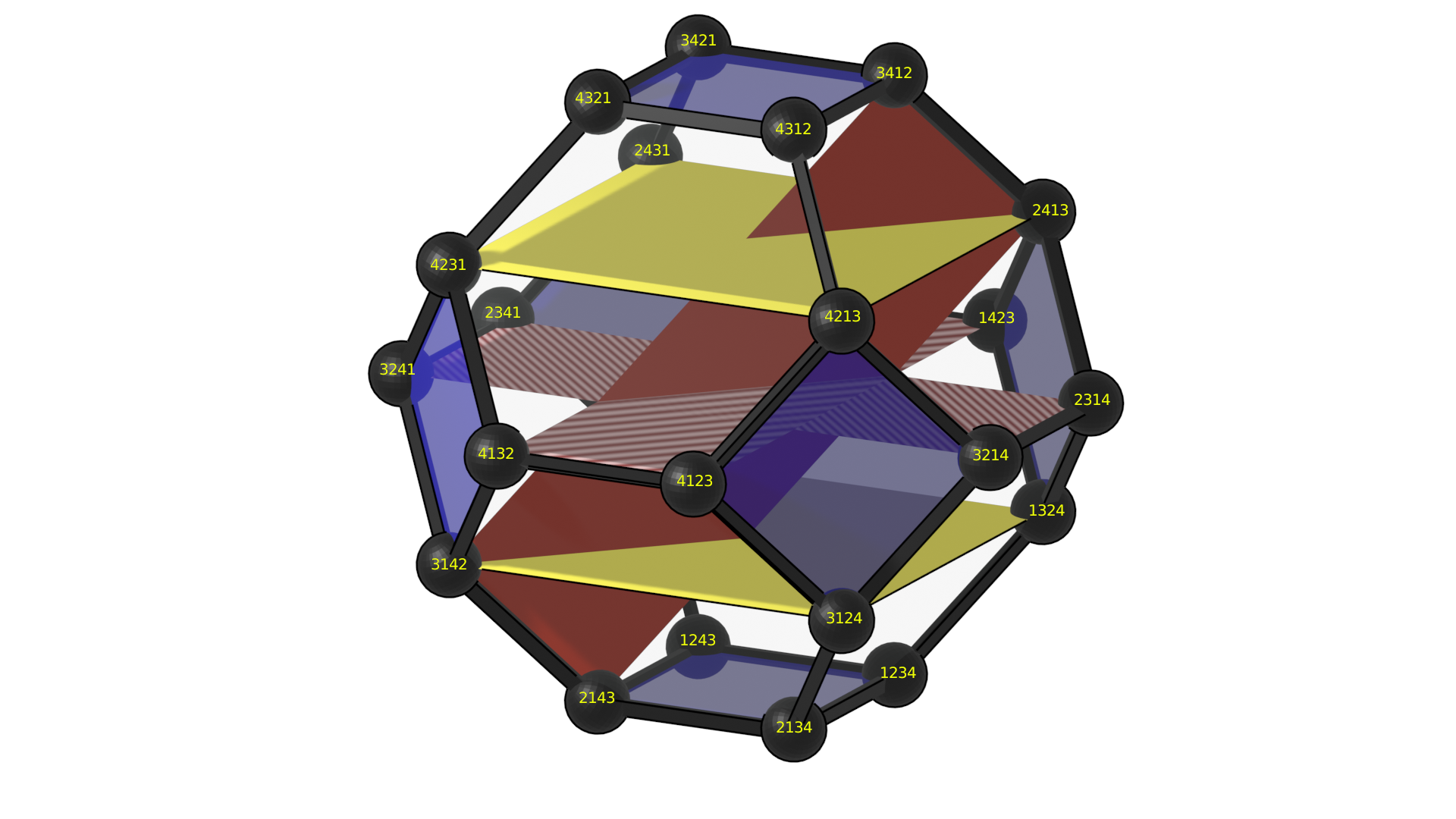}
    \caption{The 11 irreducible components $(\cp)^{2}$ of the reduced Poisson degeneracy locus in the case of $SL_{4}(\Cset)$ in the background of the Bruhat graph. The blue faces correspond to the irreducible components for the roots $\alpha_{1},\alpha_{3}$, the red ones correspond to $\alpha_{2},\alpha_{1}+\alpha_{2}+\alpha_{3}$, and yellow ones to $\alpha_{1}+\alpha_{2},\alpha_{2}+\alpha_{3}$.}
    \label{'fig:binamos'}
\end{figure}
\eex
\subsection{Non-reducedness of Poisson degeneracy loci}
\label{non-reducedness}
Next we show that the Poisson degeneracy loci 
$D_0(SL_{n+1}(\Cset)/B_+,\pi)$ are not reduced even in the simplest cases of $n=2$ and $3$. We identify
\begin{equation}
    \label{eq:ident1}
U_- \simeq U_- B_+/B_+ \quad \mbox{via} \quad
u_- \mt u_- B_+/B_+
\end{equation}
and 
\begin{equation}
    \label{eq:ident2}
\Aset^{n(n+1)/2} \simeq U_-,  
\end{equation}
where the coordinate functions on $\Aset^{n(n+1)/2}$ are denoted by $x_{ij}$ for $1 \leq j < i \leq n+1$, and the second identification is given by 
\[
(x_{ij}, 1 \leq j < i \leq n) \mt [y_{ij}]_{i,j=1}^{n+1} 
\]
with $y_{ii}:=1$, $y_{ij}:=0$ for $i<j$, $y_{ij}:= x_{ij}$ for $i>j$. Denote the vector fields
\[
\partial_{ij} := \frac{\partial}{\partial x_{ij}} \quad 
\mbox{on} \quad \Aset^{n(n+1)/2}. 
\]

\bex{ex:sl3-2}
Consider the case of $SL_3(\Cset)/B_+$. The restriction of the Poisson structure $\pi$ to $U_- B_+/B_+$, under the identifications \eqref{eq:ident1}--\eqref{eq:ident2}, is given by
\[
\pi= (-x_{21}x_{31}) \partial_{31}\wedge \partial_{21} + (x_{21}x_{32}-2x_{31}) \partial_{32}\wedge \partial_{21} + 
(-x_{31}x_{32}) \partial_{32}\wedge \partial_{31}.
\]
The ideal defining its Poisson degeneracy locus has primary decomposition
\begin{align}
\label{eq:prim-decomp1}
&(x_{21}x_{31},x_{21}x_{32}-2x_{31}, x_{31}x_{32}) =
\\
&(x_{32},x_{31}) \cap (x_{31},x_{21}) \cap 
(x_{32}^{2},x_{31}x_{32}, x_{21}x_{32}-2x_{31}, x_{21}x_{31},x_{21}^{2}). 
\nn
\end{align}
The first two ideals in the primary decomposition are prime and are precisely the vanishing ideals of the two irreducible components 
\[
D_0(SL_3(\Cset)/B_+, \pi)_{\redd} \cap U_-B_+/B_+ = 
\big( {\mathcal{R}}_{1,s_1} \cap U_-B_+/B_+ \big) \cap 
\big( {\mathcal{R}}_{1,s_2} \cap U_-B_+/B_+ \big),
\]
which match with the irreducible components of 
$D_0(SL_3(\Cset)/B_+, \pi)_{\redd}$ 
from Example \ref{ex:sl3-1}.
The third ideal in the primary decomposition is not prime; 
it defines an embedded component of the scheme and has associated prime equal to the maximal ideal $(x_{32},x_{31},x_{21})$ corresponding to the base point $B_+/ B_+ \in SL_3(\Cset)/B_+$.

A direct way of seeing that the ideal \eqref{eq:prim-decomp1} is not radical is by noticing that $x_{31}^2$ belongs to it but $x_{31}$ does not.

The flag variety $SL_3(\Cset)/B_+$ has an affine cover by Schubert cell translates
\[
\dot{v} U_- B_+/B_+ \cong U_- \cong \Aset^3
\]
for $v \in S_3$. 
Similarly to the above arguments, one can compute the restriction of the Poisson structure $\pi$ to all such open cells. This gives that the only nonreduced components of $D_0(SL_3(\Cset)/B_+, \pi)$ are 2 schemes supported at 
\[
v B_+/B_+ \quad \mbox{for} \quad v =123, 321.
\]
The corresponding vertices are colored in red on the first picture in Figure 1. 
They form a single orbit under the $\Zset_2 \times \Zset_2$-action \eqref{eq:Z2Z2} on $D_0(SL_3(\Cset)/B_+, \pi)$.
\eex
\bex{ex:sl4-2}
Consider the case of $SL_4(\Cset)/B_+$. The restriction of the Poisson structure $\pi$ to $U_- B_+/B_+$ is given by
\begin{align*}
    &-x_{31}x_{21}\partial_{31}\wedge \partial_{21} + (x_{21}x_{32} - 2x_{31})\partial_{32}\wedge \partial_{21} -x_{41}x_{21} \partial_{41} \wedge \partial_{21}
     \\ & + (x_{21}x_{42} - 2 x_{41})\partial_{42}\wedge \partial_{21} - x_{32} x_{31}\partial_{32}\wedge \partial_{31} -x_{41}x_{31} \partial_{41}\wedge \partial_{31}
      \\ & -2 x_{32}x_{41}\partial_{42}\wedge \partial_{31} + (x_{31}x_{43} - 2x_{41}) \partial_{43}\wedge \partial_{31} -x_{32}x_{42} \partial_{42}\wedge\partial_{32} \\ &+ (x_{32}x_{43} - 2x_{42}) \partial_{43}\wedge \partial_{32} 
            -x_{42}x_{41} \partial_{42}\wedge \partial_{41} - x_{43}x_{41} \partial_{43}\wedge \partial_{41} -x_{43}x_{42} \partial_{43} \wedge \partial_{42}.
\end{align*}
The ideal defining its Poisson degeneracy locus has the primary decomposition
\begin{align*}
    &( x_{42},x_{41},x_{32},x_{31}) \cap  (x_{43},x_{42},x_{41},x_{31},x_{21}) \cap 
\\&(x_{43}^2, x_{42} x_{43}, x_{41} x_{43}, x_{32} x_{43}-2 x_{42}, x_{31} x_{43}-2 x_{41},
x_{41} x_{42}, x_{32} x_{42}, x_{21} x_{42}-2 x_{41}, 
\\& \; \; x_{32} x_{41},x_{31} x_{41}, x_{21} x_{41}, x_{32}^2, x_{31} x_{32}, x_{21} x_{32}-2 x_{31}, x_{21} x_{31}, x_{21}^2).
\end{align*}    
The first two ideals in it are prime ideals, which are the vanishing ideals of the two irreducible components 
\[
D_0(SL_4(\Cset)/B_+, \pi)_{\redd} \cap U_-B_+/B_+ = 
\big( {\mathcal{R}}_{1,s_1s_3} \cap U_-B_+/B_+ \big) \cap 
\big( {\mathcal{R}}_{1,s_2} \cap U_-B_+/B_+ \big).
\]
They match with the irreducible components of 
$D_0(SL_4(\Cset)/B_+, \pi)_{\redd}$ 
from Example \ref{ex:sl4-1}.
The third ideal in the primary decomposition is not prime; 
it defines an embedded component of the scheme and has associated prime equal to the maximal ideal 
$( x_{43},x_{42}, x_{41},x_{32},x_{31},x_{21})$ 
corresponding to the base point $B_+/B_+ \in SL_4(\Cset)/B_+$.

The flag variety $SL_4(\Cset)/B_+$ has an affine cover by Schubert cell translates
\[
\dot{v} U_- B_+/B_+ \cong U_- \cong \Aset^6
\]
for $v \in S_4$. 
Similarly to the above arguments, one computes the restriction of the Poisson structure $\pi$ to all such open cells. This gives that the only nonreduced components of $D_0(SL_4(\Cset)/B_+, \pi)$ are 20 schemes supported at the points
\[
v B_+/B_+ \in SL_4(\Cset)/B_+
\]
for those $v \in S_4$ that are colored in red on the first picture in Figure 1. 
They form 7 orbits under the $\Zset_2 \times \Zset_2$-action \eqref{eq:Z2Z2} on $D_0(SL_4(\Cset)/B_+, \pi)$.
\eex
\sectionnew{Partial flag varieties}
Throughout this section we fix a parabolic subgroup $P \supseteq B_+$ of a connected simply connected complex semisimple Lie group $G$ with Levi subgroup $L \supseteq T$ and use the notations of Section~\ref{section : recollection on partial flags}. We describe the $T$-orbit stratification of the Poisson degeneracy locus of the partial flag variety $G/P$ with respect to the standard Poisson structure. We show that the strata are projected open Richardson varieties of the same combinatorial type as those in Section \ref{sec:Poisson-flag}. Their closures are shown to be isomorphic to $(\cp)^d$ for $d \in \Nset$.  The corresponding $P$-Burhat order intervals $[v,w]_P$ are shown to be isomorphic to power sets. 
\subsection{The standard Poisson structure }
The standard Poisson structure on the partial flag variety $G/P$ is given by the Poisson bivector field 
\begin{equation}
    \pi_P = \sum_{\beta \in \Delta_{+}} \chi(e_\beta) \wedge \chi(f_\beta)
\end{equation}
where $\chi: \mathfrak{g} \to \mathrm{Vect}(G/P)$ is the infinitesimal action of $G$ on $G/P$ and the root vectors $e_\beta, f_\beta$ are as in 
Section \ref{sec:Poisson-flag} (see e.g. \cite{BGY,GY}). The Poisson structure $\pi_P$ equals the push-forward of the standard Poisson structure \cite[Sect. 4.4]{ES} on $G$ under the projection map $G\to G/P$ and  
\[
\pi_P := \eta_* (\pi), 
\]
recall \eqref{eq:eta}. 
The action of the maximal torus $T$ on ($G/P,\pi_P$) is Poisson.
\bpr{p:GmodP}
For all connected simply connected complex semisimple Lie groups $G$ and parabolic subgroups $P \supseteq B_+$, the following hold:
\begin{enumerate}
\item[(i)] For all $w \in W^P$, the restriction
\[
\eta|_{B_+ w B_+/B_+}  : (B_+ w B_+/B_+, \pi) \to 
(B_+ w P/P, \pi_P)
\]
is an isomorphism of Poisson varieties. 
\item[(ii)] The $T$-orbits of symplectic leaves of $(G/P, \pi_P)$ are the projected open Richardson varieties $\Pi_v^w$ for $v \leq w \in W^P$.  
\item[(iii)] The codimension of a symplectic leaf in $\Pi_v^w$ for $v \leq w \in W^P$, i.e., the corank of $\pi$ in $\Pi_v^w$, equals
\begin{equation*}
    \dim \Ker(w^{-1}v +1).
\end{equation*}
\end{enumerate}
\epr
\begin{proof} (i) In \cite[Proposition 1.6]{GY} it is proved that the restriction
\[
\eta|_{B_- v B_+/B_+}  : (B_- v B_+/B_+, \pi) \to 
(B_- v P/P, \pi_P)
\]
is a Poisson isomorphism for all maximal length representatives $v \in W$ for the cosets in $W/W_P$. The statement in part (i) of the proposition is proved analogously.

(ii) This part was proved in \cite[Main Theorem (i)]{GY}. It also follows from part (i) of the proposition,  
Theorem \ref{t:sympl-leaves}(i) and the fact that 
the map $\eta : G/B_+ \to G/P$ is $T$-equivariant. 

(iii) This part follows at once from part (i) of the proposition and Theorem \ref{t:sympl-leaves}(ii).  
\end{proof} 
Denote
\[
\GCR_P(W) := \GCR(W) \cap (W \times W^P). 
\]
Theorem \ref{t:v-w-equiv}, Corollary \ref{c:T-orbit-stratification} and Proposition \ref{p:GmodP}(iii) imply the following:
\bco{c:GmodP} For all connected simply connected complex semisimple Lie groups $G$ and parabolic subgroups $P \supseteq B_+$, the reduced 
Poisson degeneracy locus of $(G/P, \pi_P)$ is given by
\[
D_0(G/P, \pi_P)_{\redd} = 
\bigsqcup_{(v, w) \in \GCR_P(W)} \Pi_v^w. 
\]
The strata in this stratification are precisely the 
$T$-orbits of $D_0(G/P, \pi_P)_{\redd}$. 
\eco
\subsection{The $P$-Bruhat interval $[v,w]$ for $(v,w) \in \GCR_P(W)$}
Next, for $(v,w) \in \GCR_P(W)$, we consider the $P$-Bruhat interval $[v,w]_P$ consisting of $v' \in W$ such that $v \leq_P v' \leq_P w$ 
with the partial order $\leq_P$. We show that it coincides with the Bruhat interval $[v,w]$ 
described in Theorem \ref{t:Bruhat-interval}.  

  \ble{lem: condition for P-Bruhat ordering}
    Assume  that $(v, w) \in \GCR(W)$ and $v \lessdot w$, i.e., $v< w$ and $l(w) = l(v)+1$. Denote $v = s_{\beta}w$ as in Lemma \ref{l:Bruhat}.   
    If $w^{-1}(\beta) \notin \Delta^L$, then $v \lessdot_P w$, recall \eqref{eq:Plessdot}
  \ele

  \begin{proof}
 Decompose $v$ and $w$ in a unique way as $v=v^Pv_P$ and $w = w^Pw_P$ with $v_P, w_P \in W_P$ and $v^P, w^P \in W^P$, cf. Section \ref{section : recollection on partial flags}. We have 
  $$ s_{v^{-1}(\beta)} = v^{-1}s_{\beta}v = v^{-1} w = v_P^{-1}(v^P)^{-1}w^Pw_P . $$
  Assume that the statement of the lemma does not hold. Then $v^P=w^P$, and hence 
  \[
  s_{w^{-1}(\beta)} = w^{-1} s_\beta w = w^{-1} v = w_P^{-1} v_P^{-1} \in W_P. 
  \] 
  This implies that $w^{-1}(\beta) \in \Delta^L$, which is a contradiction.
  \end{proof}
 
For the remaining of this section we fix
\[
(v,w) \in \GCR_P(W).
\]
We fix a reduced word for $w$ as in \eqref{eq:w-rw} and a subword of it with value $v$ as in \eqref{eq:v-sub-w}. We will use the notation of Section \ref{sec:nota}.

 \begin{lemma}  \label{lemma : root subgroups avoiding P} 
 For all $(v,w) \in \GCR_P(W)$ and $1 \leq k \leq d = l(w)-l(v)$, we have 
 \[
 w^{-1}(\beta_{p_k}) \notin \Delta^L
 \]
 in the notation \eqref{eq:beta-k}. 
 \end{lemma}

  \begin{proof}
   As $w \in W^P$, we have $l(ws_i) > l(w)$ for all $\alpha_i \in \Delta^L$. So $w(\alpha_i) \in \Delta_{+}$ for $\alpha_i \in \Delta^L$; that is
   $w(\Delta_{+}^L) \subset \Delta_{+}$.

   By \eqref{eq:beta-k}, $w^{-1}(\beta_j) \in - \Delta_+$, $\forall 1 \leq j \leq l(w)$. If 
   $w^{-1}(\beta_{p_k}) \in \Delta^L$ for some $1 \leq k \leq d$, then $- \beta_{p_k} \in w(\Delta_+^L) \cap (- \Delta_+)$ which is a contradiction. 
  \end{proof}

   \bpr{p:vJwK-P}
   Assume that $(v,w) \in \GCR_P(W)$ and let $d := l(w)-l(v)$.
With the notations of Section~\ref{sec:nota}, we have that  $v_J \leq_P w_K$ for every $J,K \subseteq [1,d]$ such that $J \cap K = \varnothing$. 
\epr

    \begin{proof}
    By~\eqref{eq:vJwK} we have that $(v_J,w_K) \in \GCR(W)$. Denote 
    $$ \{ l_1, \ldots , l_t \} := [1,d] \setminus (J \sqcup K). $$
    It follows from \eqref{eq:w_K}--\eqref{eq:v_K} and the commutativity of $s_{\beta_{p_1}}, \ldots, s_{\beta_{p_d}}$ that
    $$ w_K = s_{\beta_{p_{l_1}}} \ldots s_{\beta_{p_{l_t}}} v_J.  $$
 For $0 \leq r \leq t$ set
    $$ J_r =: J \sqcup \{l_1, \ldots , l_r \} .  $$ 
    Theorem \ref{t:Bruhat-interval} implies that we have the saturated chain in $W$ 
\begin{equation}
 \label{eq : conditions for lemma on P-Bruhat ordering}
v_J = v_{J_0} \lessdot v_{J_1} \lessdot \cdots \lessdot v_{J_t} = w_K . 
\end{equation}
Moreover, 
$$ v_{J_r}^{-1}(\beta_{p_{l_{r+1}}}) = v^{-1} \Big( \textstyle \prod_{j \in J} s_{\beta_{p_j}} \Big) s_{\beta_{p_{l_1}}} \ldots s_{\beta_{p_{l_r}}} (\beta_{p_{l_{r+1}}}) = v^{-1}(\beta_{p_{l_{r+1}}}) $$
 by the orthogonality of $\beta_{p_1}, \ldots , \beta_{p_d}$. Thus by Lemma~\ref{lemma : root subgroups avoiding P}, we have $v_{J_r}^{-1}(\beta_{p_{l_{r+1}}}) \notin \Delta^L$. Together with~\eqref{eq : conditions for lemma on P-Bruhat ordering}, we can apply Lemma~\ref{lem: condition for P-Bruhat ordering} to obtain that $v_{J_r} \lessdot_P v_{J_{r+1}}$ for all $0 \leq r \leq t$.
 Therefore $v_J \leq_P w_K$ by the definition of the $P$-Bruhat order. 
 \end{proof}

Theorem \ref{t:Bruhat-interval} and Proposition \ref{p:vJwK-P} imply the following:

\bth{t:P-Bruhat-interval}
For all connected simply connected complex semisimple Lie groups $G$, parabolic subgroups $P \supseteq B_+$ and 
pairs of Weyl group elements $(v,w) \in \GCR_P(W)$, the $P$-Burhat order interval $[v,w]_P$ is identical with the 
Bruhat order interval $[v,w]$ meaning that 
\[
v \leq v' \leq w' \leq w \quad \Leftrightarrow \quad v \leq_P v' \leq_P w' \leq_P w.
\]
This poset is isomorphic to the power set of $[1,l(w)-l(v)]$ with the reverse order.
\eth
\subsection{Structure of the projected closed Richardson varieties $\ol{\Pi_v^w}$ for $(v,w) \in \GCR_P(W)$}
Fix $(v,w) \in \GCR_P(W)$. We will use the notation in Section \ref{sec:nota} and Eq. \eqref{def AJK}. 
Consider the composition 
$$ \kappa_P := \eta \circ \kappa : L_{\beta_{p_1}}/B_{\beta_{p_1}} \times \cdots \times L_{\beta_{p_d}}/B_{\beta_{p_d}} \longrightarrow G/P, $$
recall \eqref{eq:eta} and \eqref{eq:kappa}. 
The following is an analogue of Propositions~\ref{prop: inclusion general case} and \ref{prop : equality}. 
\bpr{p:kappaP-injective}
      The map $\kappa_P$ is injective and for all $(v,w) \in \GCR_P(W)$,
      $$ \kappa_P(A_{J,K}) = \Pi_{v_J}^{w_K}, \quad \forall J,K \subseteq [1,d], \; \; J \cap K = \varnothing.$$
\epr

\begin{proof} Proposition~\ref{prop : equality} implies
   $$ \kappa_P(A_{J,K}) = \eta \left( \kappa(A_{J,K}) \right) = \eta \left( \mathcal{R}_{v_J}^{w_K} \right) = \Pi_{v_J}^{w_K} . $$
   Now we claim that 
\begin{equation}
\label{eq:notsim}
(v_J, w_K) \not\sim (v_{J'}, w_{K'}) 
\end{equation}
for all $J,K,J', K' \subseteq [1,d]$, such that $J \cap K = \varnothing$, $J' \cap K' = \varnothing$
and $(J,K) \neq (J',K')$. Indeed,  $(v_J, w_K) \sim (v_{J'}, w_{K'})$ means that 
$$ v_{J'}^{-1}v_J = w_{K'}^{-1}w_K \in W_P .  $$
  On the one hand, 
 $$ v_{J'}^{-1}v_J = \prod_{j \in (J \cup J') \setminus (J \cap J')} s_{v^{-1}(\beta_{p_j})},  $$
 and thus by the orthogonality of  $\beta_{p_1}, \ldots , \beta_{p_d}$, 
\[ 
v^{-1}(\beta_{p_j}) \in \Ker (1 + v_{J'}^{-1}v_J)
\]
 for all  $j \in (J \cup J') \setminus (J \cap J')$. 
 Since  $v_{J'}^{-1}v_J \in W_P$, we have
 \[
 v_{J'}^{-1}v_J = s_{\gamma_1} \cdots s_{\gamma_t} 
\]
 with $\gamma_1, \ldots \gamma_t \in \Delta^L$. Using~\eqref{eq : inclusion} we get that
\[ 
 \Ker (1+v_{J'}^{-1}v_J) \subset \Span \{ \gamma_1 , \ldots , \gamma_t \}. 
\]
Therefore,
\[ 
v^{-1}(\beta_{p_j}) \in \Span \{ \gamma_1 , \ldots , \gamma_t \}, \quad 
\forall j \in (J \cup J') \setminus (J \cap J'), 
\]
so that $v^{-1}(\gamma_{p_j}) \in \Delta^L$ which contradicts Lemma~\ref{lemma : root subgroups avoiding P}. Thus,  
\[ 
(J \cup J') \setminus (J \cap J') = \varnothing,
\]
i.e., $J=J'$.
 Similarly one shows that $K=K'$. This proves \eqref{eq:notsim}. In light of \eqref{proj3} and~\eqref{eq : stratif of proj Richardsons}, this implies that 
 $$ \kappa_P(A_{J,K}) \cap \kappa_P(A_{J',K'}) = \varnothing $$
 for all $J,K,J', K' \subseteq [1,d]$, such that $J \cap K = \varnothing$, $J' \cap K' = \varnothing$
and $(J,K) \neq (J',K')$. Hence $\kappa_P$ is injective. 
   \end{proof}

  \bth{t:kappaP-isom}
  For all connected simply connected complex semisimple Lie groups $G$, parabolic subgroups $P \supseteq B_+$ and 
 pairs of Weyl group elements $(v,w) \in \GCR_P(W)$, the map $\kappa_P$ defines an isomorphism of algebraic varieties 
 $$ \overline{\Pi_v^w} \simeq (\cp)^{l(w)-l(v)} . $$
 Under this isomorphism, the orbits for the natural $(\Cset^\times)^d$-action on $(\cp)^d$ match with the projected open Richardson varieties $\Pi_{v'}^{w'}$ for $v \leq v' \leq w' \leq w$.
  \eth
  \begin{proof} It follows from Propositions \ref{prop : equality} and \ref{p:kappaP-injective}, and Eq. \eqref{eq : stratif of proj Richardsons} that $\kappa_P$ defines a bijection between  
  $L_{\beta_{p_1}}/B_{\beta_{p_1}} \times \cdots \times L_{\beta_{p_d}}/B_{\beta_{p_d}}$ and $\overline{\Pi}_v^w$. 
  Analogously to the proof of Theorem \ref{t:kappa-isom}, 
  one shows that the differential of $\kappa_P$ is everywhere injective. This proves that $\kappa_P$ defines an isomorphism 
  \[
  L_{\beta_{p_1}}/B_{\beta_{p_1}} \times \cdots \times L_{\beta_{p_d}}/B_{\beta_{p_d}} \simeq \overline{\Pi}_v^w
  \]
  see e.g. \cite[Corollary 14.10]{Harris}. The last statement of the theorem follows from Proposition \ref{prop : equality}, Theorem \ref{t:P-Bruhat-interval} and Eq. \eqref{eq:projopenRich}.
  \end{proof}
\sectionnew{The dimension of reduced Poisson degeneracy loci}

As shown in Proposition \ref{p:irr-comp}, the reduced Poisson degeneracy locus $(G/B_+, \pi)_{\redd}$ of any full flag variety are isomorphic to copies of $(\cp)^d$ for integers $d$ that are not necessarily the same, see Examples \ref{ex:sl3-1} and \ref{ex:sl4-1}. In this section, we use Kostant's cascades of roots \cite{Kostant} to compute the top dimension of the ireducible components of the reduced Poisson degeneracy locus proving the following result: 

 \bth{thm : main thm dimension of X}
  For any connected simply connected complex semisimple Lie group $G$,
  the top-dimensional irreducible components of the reduced degeneracy locus $D_0(G/B_+, \pi)_{\redd}$ are isomorphic to $(\cp)^m$, where 
  \[
  m= l_{\Delta}(w_0) = \sharp \mathcal{B}
  \]
  and $\mathcal{B}$ is Kostant's cancade of roots of $\g$. 
 \eth

 The proof of this result splits into two main steps carried out in Sections~\ref{section : maximality} and~\ref{section : existence}. We first show that $l_{\Delta}(w_0)$ is an upper bound for the dimension of an irreducible component of $D_0(G/B_+, \pi)_{\redd}$ and then construct an irreducible component that has dimension equal to $l_{\Delta}(w_0)$. 
 
\subsection{Recollection on Kostant's cascade of roots}
\label{section : recollection on cascades}
We begin with a brief review of the components of Kostant's construction \cite{Kostant} of the cascade of roots of a semisimple Lie algebra $\g$. (Kostant's original formulation is for a simple Lie algebra $\g$, but the same method applies to semisimple Lie algebras.) For a positive root 
\[
\beta = \textstyle \sum_i k_i \alpha_i \in \Delta_+, \quad k_i \in \Nset,
\]
denote by $\Delta(\beta)$ the (indecomposable) root subsystem with simple roots
\[
\{ \alpha_i \mid 1 \leq i \leq n, k_i \neq 0 \}.  
\]
The root $\beta$ is called locally high if it is the highest root of $\Delta(\beta)$. Denote by $\g(\beta)$ the simple 
$\ad(\h)$-stable Lie subalgebra of $\g$ with root system $\Delta(\beta)$. The root subsystem of $\Delta(\beta)$ 
\[
\Delta(\beta)^\circ := \pm \{ \gamma \in \Delta(\beta) \mid (\gamma, \beta) = 0 \}  
\]
is in general decomposable. Its highest roots are call {\em{descendants}} of $\beta$. Denote also 
\[
E(\beta) := \{ \gamma \in \Delta(\beta) \mid (\gamma , \beta)>0\}.
\]

Kostant's cascade of roots of $\g$ is the collection $\mathcal{B}$ of all positive roots of $\g$ obtained 
as iterative descendants of the highest roots of the simple components of $\g$ (including the latter). The main property of this collection of roots is the following: 

\begin{theorem}[Kostant, \cite{Kostant}] \label{thm : Kostant}
The set $\mathcal{B}$ is a maximal set of pairwise strongly orthogonal positive roots containing the highest root of each simple component of $\mathfrak{g}$. Furthermore we have that 
$$ w_0 = \prod_{\gamma \in \mathcal{B}} s_{\gamma} \quad \text{and} \quad \Delta_{+} = \bigsqcup_{\gamma \in \mathcal{B}} E(\gamma). $$
\end{theorem}
By Lemma \ref{lem : refl length}, the reflective length of the longest element of $W$ is 
\[
l_\Delta(w_0) = \sharp \mathcal{B}.
\]

\begin{lemma}[Kostant, {{\cite[Propositions 1.1 and 1.10]{Kostant}}}] \label{lem : Kostant}
The following properties hold for all $\gamma \in \mathcal{B}$:
\begin{enumerate}
       \item[(i)] The set $E(\gamma)$ contains $\gamma$ and $\sharp E(\gamma) = 2h^{\vee}(\gamma)-3$, where $h^{\vee}(\gamma)$ denotes the dual Coxeter number of $\mathfrak{g}(\gamma)$.
       \item[(ii)] For every $\mu \in E(\gamma) \setminus \{ \gamma\}$, we have $2(\gamma,\mu)/ \| \gamma\|^2 = 1$.
       \item[(iii)] For every $\mu \in E(\gamma) \setminus \{\gamma\}$, there exists a unique $\nu \in E(\gamma) \setminus \{\gamma\}$ such that $\mu + \nu = \gamma$. 
\end{enumerate}
\end{lemma}

Given $\gamma \in \mathcal{B}$, the pairs $(\mu, \nu) \in E(\gamma) \setminus \{\gamma\}$ such that $\mu + \nu = \gamma$ are called {\em{Heisenberg pairs}}, and two positive roots $\mu$ and $\nu$ forming such a pair are called {\em{Heisenberg twins}}; automatically $\mu \neq \nu$. The following observation will be useful to us. 

 \ble{lem : convexity and cascade}
     Let $\gamma \in \mathcal{B}$ and $\Psi$ be a convex subset of $\Delta_+$ (see Section~\ref{subsection : W gp and roots}) such that $\gamma \notin \Psi$. Then the following hold:
      \begin{enumerate}
          \item[(i)] The set $\Psi$ contains at most one element of any Heisenberg pair in $E(\gamma)$. 
          \item[(ii)] If moreover one has $\Psi \subset E(\gamma) \setminus \{\gamma \}$ and $\sharp \Psi = h^{\vee}(\gamma)-2$, then $\Psi$ consists of exactly one element of each Heisenberg pair in $E(\gamma)$. 
      \end{enumerate}
 \ele

\begin{proof}
 Assume $\Psi$ contains two Heisenberg twins $\mu$ and $\nu$ in $E(\gamma)$. Then by convexity of $\Psi$ we have that 
 $$ \gamma = \mu + \nu \in \Psi $$
 which contradicts our assumption. Thus (i) holds. If moreover $\Psi \subset E(\gamma) \setminus \{\gamma \}$ and $\sharp \Psi = h^{\vee}(\gamma)-2$, then Lemma~\ref{lem : Kostant} (i) implies that $\sharp \Psi$ is equal to the number of Heisenberg pairs in $E(\gamma)$, and thus $\Psi$ contains exactly one element of each Heisenberg pair in $E(\gamma)$ which proves (ii). 
 \end{proof}

     \subsection{Maximality of $l_{\Delta}(w_0)$}
     \label{section : maximality}

     In this subsection we prove the following property of the reflection lengths of involutive elements of $W$. 

     \bpr{prop : maximality}
 Let $w \in W$ be such that $w^2=1$. Then one has 
 \[
 l_{\Delta}(w) \leq l_{\Delta}(w_0).
 \]
     \epr

 \begin{proof}
 We consider several cases according to classical properties of the longest element of finite type Weyl groups. We will be using the labeling of simple roots from {{\cite[Chap. VI, §4, no.5-13]{Bourbaki}}}. 

 Assume $\mathfrak{g}$ is of one of the following types: $B_n, n \geq 1, C_n, n \geq 1, D_n, n \geq 4$ with $n$ even, $E_7, E_8, F_4$ or $G_2$. In all these cases $w_0$ is known to be equal to $-1$, see e.g. \cite[Chap. VI, \S4]{Bourbaki}. Hence $\dim \Ker(1-w_0)=0$ which implies $\l_{\Delta}(w_0)=n$ by Lemma~\ref{lem : refl length}. In particular, for any $w \in W$, Lemma~\ref{lem : refl length} yields $l_{\Delta}(w) \leq n = l_{\Delta}(w_0)$. 
 
  Assume that $\mathfrak{g}$ is of type $D_n$ with $n \geq 4$ odd. Then $w_0$ is equal to the involution acting on $\Pi$ as follows:
  $$ w_0(\alpha_n) = - \alpha_{n-1} \quad w_0(\alpha_{n-1}) = - \alpha_n \quad w_0(\alpha_i)=-\alpha_i, 1 \leq i \leq n-2 . $$
  Therefore, $\dim \Ker(1-w_0)=1$ and thus $l_{\Delta}(w_0)=n-1$ by Lemma~\ref{lem : refl length}. Assume there exists $w \in W$ such that $w^2=1$ and $l_{\Delta}(w) > l_{\Delta}(w_0)$. Then necessarily $l_{\Delta}(w) = n$ and hence $\dim \Ker(1-w)=0$ by Lemma~\ref{lem : refl length}. As $w$ is an involution, this implies that $w=-1$. In particular we have that $-1 \in W$ which is known to be false for $\mathfrak{g}$ of type $D_n$ with $n$ odd, see \cite[Chap. VI, \S4, no.8]{Bourbaki}).

  Assume next that $\mathfrak{g}$ is of type $A_n , n \geq 1$. In this case $w_0$ is the involution  acting on $\Pi$ as follows:
  $$ w_0(\alpha_i) = - \alpha_{i^{*}} \qquad \text{where $i^{*} := n-i+1$ for each $1 \leq i \leq n$.} $$ 
  Hence we get $\dim \Ker(1+w_0) = \lceil n/2 \rceil$. Let us now fix $w \in W$ such that $w^2=1$. Viewing $w$ as a permutation of the set $\{1, \ldots , n+1\}$, we can write $w$ in a unique way as a product of disjoint cycles. Because $w^2=1$, these cycles have to be transpositions, and as their supports are disjoint, there can be at most $\lceil n/2 \rceil$ of them. Thus we have that $w$ can be written as a product of at most $l_{\Delta}(w_0)$ pairwise commuting reflections, and hence $l_{\Delta}(w) \leq l_{\Delta}(w_0)$.

  Assume finally that $\mathfrak{g}$ is of type $E_6$. Then $w_0$ is the involution acting on simple roots as 
   $$ w_0(\alpha_1)=-\alpha_6 \quad w_0(\alpha_3)=-\alpha_5 \quad w_0(\alpha_4)=-\alpha_4 \quad w_0(\alpha_2)=-\alpha_2 . $$
   Hence one has $\dim \Ker(1-w_0)=2$ and thus $l_{\Delta}(w_0)=4$ by Lemma~\ref{lem : refl length}. Thus we shall prove that any involution in $W$ has reflection length at most $4$.  
   By Lemma~\ref{lem : refl length}, $w$ can be written as a product of reflections associated to pairwise orthogonal roots, and up to conjugating (which does not change the reflection length) we may assume that one of them is the highest root $\theta$. The subset of roots orthogonal to $\theta$ can be identified with a type $A_5$ root system, see e.g. \cite[Chap. VI, \S4, no.12]{Bourbaki}. Hence, by what has been proved above in type $A_n$, we get that a collection of pairwise orthogonal roots that are all orthogonal to $\theta$ is necessarily of cardinality at most $3$ (as by Lemma~\ref{lem : refl length}, the cardinality of such collection of roots is the reflection length of the product of the corresponding reflections). Thus we obtain that $w$ is of reflection length at most $4$. 
\end{proof}

\bco{coro : maximality}
    For any pair $(v,w) \in \GCR(W)$, the Richardson variety $\overline{\mathcal{R}_v^w}$ is of dimension at most $l_{\Delta}(w_0)$. 
\eco

 \begin{proof}
Let $v,w \in \GCR(W)$. Condition (4) in Theorem \ref{t:v-w-equiv} implies that $(v w^{-1})^2 =1$ and 
\[
l_\Delta(v w^{-1}) = l(w) - l(v) = \dim \Rvw.
\]
Applying Proposition~\ref{prop : maximality}, we
obtain 
\[
\dim \Rvw = l_\Delta(v w^{-1}) \leq l_\Delta(w_0).
\]
 \end{proof}

\subsection{Existence of an irreducible component of the reduced Poisson degeneracy locus of dimension $l_{\Delta}(w_0)$}
   \label{section : existence}
   
 The rest of this section will be devoted to proving the existence of a pair 
 \[
 (v,w) \in \GCR(W) \quad \mbox{such that} 
 \quad l(w)-l(v)=l_{\Delta}(w_0)  
 \]
 which completes the proof of Theorem \ref{thm : main thm dimension of X}. 
 We begin by providing a simple sufficient condition on $v$ guaranteeing the existence of $w$ such that the pair $(v,w)$ fulfills these requirements. In all what follows we set 
 \[
 m := l_{\Delta}(w_0), \quad N:= l(w_0) = \sharp \Delta_+.
 \]
   
    \ble{lem : conditions on v}
   Let $v \in W$. Assume that $l(v) = (N-m)/2$ and $\mathcal{B} \cap \Delta_+^v = \varnothing$. Set $w := w_0v$. Then, $(v,w) \in \GCR(W)$ and $l(w)-l(v)=m$. 
     \ele
     
      \begin{proof}
 Since $l(w_0v)=N-l(v)$, we have 
     $$ l(w)-l(v)  = N-2l(v) = N-(N-m)=m. $$
     
To prove that $(v,w) \in \GCR(W)$, we will show that the pair $(v,w)$ satisfies condition (5) in Theorem \ref{t:v-w-equiv}. Firstly, from
     Theorem~\ref{thm : Kostant} we get $w = \prod_{\beta \in \mathcal{B}} s_{\beta} v$ with the roots in $\mathcal{B}$ being pairwise orthogonal. Moreover, by Lemma~\ref{lem : refl length} we have that $m= \sharp \mathcal{B}$, and hence $l(w)-l(v)= m= \sharp \mathcal{B}$. In order to complete the proof, it remains to show that 
     $v \leq w$. Since $w=w_0v$, we have $\Delta_+^w= \Delta_+ \setminus \Delta_+^v$. In particular our assumptions on $v$ imply $\mathcal{B} \subseteq \Delta_+^w$. Choose a reduced word $(i_1, \ldots, s_{i_l})$, where $l := l(w)= (N+m)/2$. There exit indices $1 \leq p_1 < p_2 < \cdots < p_m \leq l$ such that 
    $$ \mathcal{B} = \{ \beta_{p_k} \mid 1 \leq k \leq m \} $$
    in the notation \eqref{eq:beta-k}. 
  Using Lemma~\ref{l:Bruhat} and Theorem~\ref{thm : Kostant} we obtain
\begin{align*}
v = w_0 w &= 
\Big( \textstyle \prod_{\gamma \in \mathcal{B}} s_{\gamma} \Big) s_{i_1} \ldots s_{i_l} 
\\
&=\big( s_{\beta_{p_1}} \ldots s_{\beta_{p_m}} \big) s_{i_1} \ldots s_{i_l}
= s_{i_1} \ldots  \widehat{s}_{i_{p_1}} \ldots \widehat{s}_{i_{p_m}} \ldots s_{i_l}, 
\end{align*}
which shows that $v \leq w$. 
\end{proof}

 The rest of the proof of Theorem \ref{thm : main thm dimension of X} amounts to constructing an element $v \in W$ satisfying the conditions of Lemma~\ref{lem : conditions on v}. More precisely, we will construct an element $v$ such that $\Delta_+^v$ contains exactly one element of each Heisenberg pair in $\Delta_+ \setminus \mathcal{B}$.

\bpr{prop : constructing v for theta} Let $\g$ be a complex simple Lie algebra and $\theta$ be its highest root. 
There exists $u \in W$ such that $\Delta_+^u$ contains exactly one member of each Heisenberg pair in $E(\theta) \setminus \{ \theta \}$. 
 \epr

 \begin{proof} We will construct a Weyl group element $u$ such that 
 \begin{equation}
 \label{eq:dualCoxeter}
 l(u) = h^{\vee}-2 
 \quad \mbox{and} \quad 
 \Delta_+^u \subset E(\theta)\backslash \{ \theta\}.
 \end{equation}
 Since $\sharp \Delta_+^u = l(u)$, Lemma~\ref{lem : convexity and cascade} implies that for this $u \in W$, $\Delta_+^u$ contains exactly one element of every Heisenberg pair in $E(\theta)\backslash \{ \theta\}$.
We provide a uniform proof when $\g$ is of simply laced type. The non-simply laced cases will be dealt with on a case by case investigation. 

 Assume that $\g$ is simply laced and recall from Section \ref{sec:background-Lie} that $n$ denotes the rank of $\g$. Recall also that $\mathrm{ht}(\theta) = h-1$, where $h$ denotes the Coxeter number of $\g$ (see for instance  {{\cite[Chap. VI, §1, Proposition 31]{Bourbaki}}}). 
 We will show that there exists a sequence of indices $1 \leq i_1, \ldots , i_{h-2} \leq n$ such that 
 $$ s_{i_k} \ldots s_{i_1}(\theta) = \theta - (\alpha_{i_1} + \cdots + \alpha_{i_k}) $$
 for every $1 \leq k \leq h-2$. We argue by induction on $k$. 
  Obviously, there exists $1 \leq i_1 \leq n$ such that $\theta$ is not orthogonal to the simple root $\alpha_{i_1}$. Then $\alpha_{i_1} \in E(\theta)$, and hence, $s_{i_1}(\theta) = \theta - \alpha_{i_1}$ by Lemma~\ref{lem : Kostant}(ii). Assume that $i_1, \ldots , i_k$ have been constructed as required for some $k < h-2$, and denote $\beta := s_{i_k} \ldots s_{i_1}(\theta)$. We have $\mathrm{ht}(\beta)= h-k-1 >1$ so that $\beta$ is not a simple root. In particular, $(\alpha_i, \beta) \neq  \pm 2$ and hence
\[  
  (\alpha_i, \beta) \in \{-1,0,1\}
\]  
for all $1 \leq i \leq n$. 
As $\|\beta \|>0$, there exists $1 \leq i \leq n$ such that $\alpha_i \in \Delta_+(\beta)$ and $(\alpha_i , \beta)>0$. Denoting $i_{k+1}:=i$, we then have $(\alpha_{i_{k+1}}, \beta) = 1$ so that $s_{i_{k+1}}(\beta) = \beta - \alpha_{i_{k+1}}$. 

   Set $u := s_{i_1} \ldots s_{i_{h-2}}$. For $1 \leq k \leq h-2$,
   $$  \big( \theta, s_{i_1} \ldots s_{i_{k-1}}(\alpha_{i_k}) \big) = \big( s_{i_{k-1}} \ldots s_{i_1}(\theta), \alpha_{i_k}  \big)  >0 $$
   by the construction of $i_1, \ldots , i_{h-2}$. This implies that $s_{i_1} \ldots s_{i_{k-1}}(\alpha_{i_k}) \in \Delta_+$
    because for $\gamma \in \Delta$, we have  $(\theta, \gamma)>0$ $\Leftrightarrow$  $\gamma \in \Delta_+$. Thus $s_{i_1} \ldots s_{i_{k-1}}(\alpha_{i_k}) \in \Delta_+$ for each $1 \leq k \leq h-2$ hence $(i_1, \ldots, i_{h-2})$ is a reduced word.

     The above shows that $(\theta, \beta)>0$ for all $\beta \in \Delta_+^u$. In other words $\Delta_+^u \subset E(\theta)$. Moreover, if $\theta$ was in $\Delta_+^u$, then there would exist $1 \leq k \leq h-2$ such that $\theta = s_{i_1} \ldots s_{i_{k-1}}(\alpha_{i_k})$ so that $s_{i_{k-1}} \ldots s_{i_1}(\theta) \in \Pi$ which is impossible as $\mathrm{ht}(s_{i_{k-1}} \ldots s_{i_1}(\theta)) = \mathrm{ht}(\theta)-k+1 > 1$. Hence $\theta \notin \Delta_+^u$, so $\Delta_+^{u} \subset E(\theta) \setminus \{ \theta\}$. 
     Since $l(u) = h-2 = h^{\vee}-2$, $u$ satisfies \eqref{eq:dualCoxeter}.
     \smallskip

    If $\g$ is of type $C_n, n \geq 1$, set $u := s_1 \ldots s_{n-1}$. The word $(1, \ldots, n-1)$ is reduced of length $n-1 = h^{\vee}-2$ and
    $$ \Delta_+^u = \{ \alpha_1 + \cdots + \alpha_k , 1 \leq k < n \}$$
    from which one easily checks that $\Delta_+^u \subset E(\theta) \setminus \{ \theta\}$, so $u$ satisfies \eqref{eq:dualCoxeter}.

     If $\g$ is of type $B_n, n \geq 1$, set $u = s_2 \ldots s_ns_1 \ldots s_{n-2}$. The word $(2, \ldots, n, 1, \ldots, n-2)$ is reduced of length $2n-3 = h^{\vee}-2$ and
    $$ \Delta_+^u = \{ \alpha_1 + \cdots + \alpha_j , 1<j<n \} \sqcup \{ \alpha_2 + \cdots \alpha_j , 2 \leq j \leq n \}.$$
    It is straightforward to check that $\Delta_+^u \subset E(\theta) \setminus \{ \theta\}$, so $u$ satisfies \eqref{eq:dualCoxeter}. 

    If $\g$ is of type $G_2$, set $u=s_1s_2$ so that $\Delta_+^u = \{\alpha_1, \alpha_1+ \alpha_2\} \subset E(\theta) \backslash \{ \theta \}$ and $u$ satisfies \eqref{eq:dualCoxeter}.

    If $\g$ is of type $F_4$, sets $u = s_1s_2s_3s_4s_2s_3s_1$; $l(u) = 7 = h^{\vee}-2$ and
    $$ \Delta_+^u = \{ 1000, 1100, 1110, 1111, 1120, 1121, 1220 \}, $$
    where $abcd$ stands for $a \alpha_1 + b \alpha_2 + c \alpha_3 + d \alpha_4$. Again, one easily checks that $\Delta_+^u \subset E(\theta) \setminus \{ \theta\}$, which show that $u$ 
    staisfies \eqref{eq:dualCoxeter}.
 \end{proof}

 \bpr{prop: existence}
For each semisimple Lie algebra $\mathfrak{g}$, there exists $v \in W$ such that $\Delta_{+}^v$ consists of one element of each Heisenberg pair in $\Delta_+$. 
 \epr

 \begin{proof}
 Let us fix a simple component $\mathbf{g}$ of $\mathfrak{g}$ and denote $\Delta := \Delta(\mathbf{g}), \Delta_+ := \Delta_+(\mathbf{g})$. Let $\theta$ denote the highest root of $\mathbf{g}$ and $\mathcal{B} := \{ \gamma_1, \ldots , \gamma_m \}$ be the cascade of roots of $\mathbf{g}$, with $\gamma_1 = \theta$, the highest root of $\mathbf{g}$. Let $u$ denote the element of the Weyl group of $\mathbf{g}$ provided by Proposition~\ref{prop : constructing v for theta}. We first prove the following :

  \ble{lem : cascade for g'}
  Let $\mathfrak{g}'$ denote the Lie subalgebra of $\mathbf{g}$ generated by the elements $h_{\beta}, e_{ \pm \beta}$ for all $\beta \in \Delta_+$ such that $ \beta \bot u^{-1}(\theta)$. Then $\g'$ is a semisimple Lie algebra and the set $\{u^{-1}(\gamma_k) , 2 \leq k \leq m \}$ is the cascade of roots of $\g'$.  
\ele
 
 \begin{proof}
 The set $\Delta'$ defined as 
 $$ \Delta' := \{ \beta \in \Delta \mid \beta \bot u^{-1}(\theta) \} $$
 defines a root system in the subspace of $\mathfrak{t}^{*}$ generated by the elements of 
 $$ \Pi' := \{ \beta \in \Delta_+ \cap \Delta' \mid \forall \gamma \in \Delta_+, \gamma < \beta \Rightarrow (\gamma , u^{-1}(\theta)) \neq 0 \} . $$
  Therefore the Lie subalgebra $\mathfrak{g}'$ of $\mathbf{g}$ is semisimple and its sets of roots, positive roots and simple roots are respectively given by  $\Delta'$, $\Delta'_+  := \Delta_+ \cap \Delta'$ and $\Pi'$. Moreover we have that 
  \begin{equation} \label{eq : inclusion}
 u(\Delta'_+) \subset \Delta_+.
  \end{equation} 
  Indeed, if $\beta \in \Delta'_+$ is such that $u(\beta) \in -\Delta_+$ then $u(\beta) \in (-\Delta_+) \cap u(\Delta_+) = - \Delta_{+}^{u} \subset - E(\theta)$ by the construction of $u$, and hence $(u(\beta), \theta) <0$, which contradicts $\beta \bot u^{-1}(\theta)$. 
  
Denote by $\mathcal{B}' := \{ \gamma'_1, \ldots , \gamma'_{m'}\}$ the cascade of roots of $ \mathfrak{g}'$ and by $E'(\gamma'_i)$ the corresponding sets as defined in Section~\ref{section : recollection on cascades}. 
We now prove that 
$$ \mathcal{B}' \subseteq u^{-1}(\mathcal{B} \setminus \{ \theta \}).$$
Assume that $\gamma'_{k_1} \notin u^{-1}( \mathcal{B} \setminus \{ \theta \})$ for some $1 \leq k_1 \leq m'$. Thus $u(\gamma'_{k_1}) \notin \mathcal{B}$; it cannot be equal to $\theta$ since $\gamma'_{k_1} \bot u^{-1}(\theta)$ by definition. Moreover, from \eqref{eq : inclusion} we obtain $u(\gamma'_{k_1}) \in \Delta_+$, so $u(\gamma'_{k_1}) \in \Delta_+ \setminus \mathcal{B}$. Hence by Theorem~\ref{thm : Kostant}, we have
 $$ u(\gamma'_{k_1}) + \nu_1 = \gamma_{l_1}  $$
for some $l_1 \geq 2$ ($l_1 \neq 1$ because $u(\gamma'_{k_1})$ is orthogonal to $\theta$ and thus cannot belong to $E(\theta)$) where $\nu_1 \in \Delta_+$ is the Heisenberg twin of $u(\gamma'_{k_1})$. This yields
 \begin{equation} \label{eq : Heisenberg1}
 \gamma'_{k_1} + u^{-1}(\nu_1) = u^{-1}(\gamma_{l_1}), 
 \end{equation}
 and as $l_1 \neq 1$, we have that $E(\gamma_{l_1}) \cap \Delta_+^{u} = \varnothing$ because $\Delta_+^{u} \subset E(\gamma_1)$ by the construction of $u$.  Hence $u^{-1}(E(\gamma_{l_1})) \subset \Delta_+$, so $u^{-1}(\nu_1)$ and $u^{-1}(\gamma_{l_1})$ belong to $\Delta_+$. 

 In particular, we have that $u^{-1}(\gamma_{l_1})-\gamma'_{k_1} \in \Delta$, which implies that $u^{-1}(\gamma_{l_1}) \notin \mathcal{B}'$ because the elements of $\mathcal{B}'$ are strongly orthogonal, see Theorem~\ref{thm : Kostant}. So $u^{-1}(\gamma_{l_1})$ is a positive root; it is orthogonal to $u^{-1}(\theta)$ because $(\gamma_{l_1}, \theta)=0$ (recall that $l_1 \neq 1$)
 and $u^{-1}(\theta) \notin \mathcal{B}'$. In other words, we have 
 $$ u^{-1}(\gamma_{l_1}) \in \Delta'_+ \setminus \mathcal{B}', $$
 so similarly as above, we have 
 \begin{equation} \label{eq : Heisenberg2}
  u^{-1}(\gamma_{l_1}) + \delta_1 = \gamma'_{k_2} 
  \end{equation}
  for some $k_2 \in \{1 , \ldots , m'\}$, where $\delta_1 \in \Delta'_+$ is the Heisenberg twin of $u^{-1}(\gamma_{l_1})$ in $E'(\gamma'_{k_2})$. In particular combining~\eqref{eq : Heisenberg2} with~\eqref{eq : Heisenberg1} we get 
 $$ \gamma'_{k_1} + u^{-1}(\nu_1) + \delta_1 = \gamma'_{k_2} $$
 so that $\gamma'_{k_1} < \gamma'_{k_2}$. Moreover, $\gamma'_{k_2} \notin u^{-1}(\mathcal{B})$ because $\gamma'_{k_2} - u^{-1}(\gamma_{l_1}) \in \Delta$ so this would contradict the strong orthogonality of the elements of $\mathcal{B}$. Therefore we can repeat the same arguments and eventually we end up constructing a sequence $k_1, k_2, \ldots $ with $\gamma'_{k_1} < \gamma'_{k_2} < \cdots $ which is a contradiction as $\mathcal{B}'$ is finite. 
 This proves that $\mathcal{B}' \subseteq u^{-1}(\mathcal{B} \setminus \{ \theta \})$. By the maximality of $\mathcal{B}'$ (see Theorem~\ref{thm : Kostant}), this inclusion must be an equality, which proves the lemma. 
 \end{proof}
 
By induction, there exists an element $v' \in W' := W(\mathfrak{g}')$ such that $\Delta_+^{v'}$ consists of one element of each Heisenberg pair in $\Delta'_+$. We now prove the following: 

 \ble{lem : induction for H pairs}
The set $u(\Delta_{+}^{v'})$ consists of one element of each Heisenberg pair in $\Delta_+ \setminus E(\theta)$. 
\ele

\begin{proof}
 Let $\mu \in \Delta_{+}^{v'}$ and $\nu$ be its Heisenberg twin in $\Delta'_+$. By Lemma~\ref{lem : cascade for g'}, we have $k \geq 2$ such that $\mu + \nu = u^{-1}(\gamma_k)$ and thus $u(\mu) + u(\nu) = \gamma_k$. Moreover, by~\eqref{eq : inclusion}, both $u(\mu)$ and $u(\nu)$ belong to $\Delta_+$, so they are Heisenberg twins in $E(\gamma_k)$. Hence $u(\mu)$ is an element of a Heisenberg pair in $\Delta_+ \setminus E(\theta)$. 

  Conversely, given a Heisenberg pair $(\mu , \nu)$ in $E(\gamma_k)$, for any $k \geq 2$  we have that $(\theta,\mu)=(\theta,\nu)=0$ and hence $u^{-1}(\mu)$ and $u^{-1}(\nu)$ are in $\Delta'$. In fact, they are in $\Delta'_+$. Indeed, 
  $$ u^{-1}(\mu) \in  u^{-1} \left( \Delta_+ \setminus E(\theta) \right)  \subset u^{-1} \left( \Delta_+ \setminus \Delta_{+}^u \right) \subset \Delta_+, $$
  and similarly for $\nu$. As $u^{-1}(\mu) + u^{-1}(\nu) = u^{-1}(\gamma_k)$, we get that $(u^{-1}(\mu),u^{-1}(\nu))$ is a Heisenberg pair in $E'(u^{-1}(\gamma_k))$. Therefore, one element of this pair say $u^{-1}(\mu)$ belongs to $\Delta_{+}^{v'}$ which yields $\mu \in u(\Delta_{+}^{v'})$. This proves that at least one element of the Heisenberg pair $(\mu, \nu)$ belongs to $u(\Delta_{+}^{v'})$, and this holds for each pair in $E(\gamma_k)$ for any $2 \leq k \leq m$. 
  
   Finally, $u(\Delta_{+}^{v'})$ is convex and by Lemma~\ref{lem : cascade for g'}, $u^{-1}(\gamma_k) \in \mathcal{B}'$ for each $k \geq 2$, so  $\gamma_k \notin  u(\Delta_{+}^{v'})$ as $\Delta_{+}^{v'} \cap \mathcal{B}' = \varnothing$ (by construction of $v'$). Hence Lemma~\ref{lem : convexity and cascade}(i) implies that  $u(\Delta_{+}^{v'})$ contains at most one element of each Heisenberg pair in $E(\gamma_k)$, and this holds for each $k \geq 2$ so we obtain that $u(\Delta_{+}^{v'})$   cannot contain two Heisenberg twins in $\Delta_+ \setminus E(\theta)$.
   This finishes the proof of the lemma.
\end{proof}

   To conclude the proof of Proposition \ref{prop: existence}, we set $v := uv'$. We have that 
$$ \Delta_{+}^{v} = \Delta_{+}^u \sqcup u(\Delta_{+}^{v'}) $$
where the union is disjoint because 
$$ u(\Delta_{+}^{v'}) \subset \Delta_+ \setminus E(\theta) \quad \text{and} \quad \Delta_{+}^u \subset E( \theta) $$
by Lemma~\ref{lem : induction for H pairs} and by the construction of $u$. Furthermore, $\Delta_{+}^{u}$ (resp. $u(\Delta_{+}^{v'})$) contains exactly one element of each Heisenberg pair in $E(\theta) = E(\gamma_1)$ (resp. in $E(\gamma_2) \sqcup \ldots \sqcup E(\gamma_m)$). So we get that $\Delta_{+}^{v}$ contains exactly one element of each Heisenberg pair in $\Delta_+(\mathbf{g})$.

Therefore, denoting by $\mathbf{g}_1, \ldots , \mathbf{g}_k$ the simple components of $\mathfrak{g}$, we obtain for every $1 \leq i \leq k$ an element $v_i \in W(\mathbf{g}_i)$ such that $\Delta_{+}^{v_i}$ contains exactly one element of each Heisenberg pair in $\Delta_+(\mathbf{g}_k)$. 
Hence setting $v := v_1 \ldots v_k$, we have that $v \in W(\mathfrak{g})$ and $\Delta_{+}^v = \bigsqcup_i \Delta_{+}^{v_i}$ contains exactly one element of each Heisenberg pair in $\Delta_+$.
\end{proof}

We can now prove the main result of this section. 

 \begin{proof}[Proof of Theorem~\ref{thm : main thm dimension of X}]
Let $v$ denote the element of $W$ constructed in Proposition~\ref{prop: existence}. In particular, $\Delta_{+}^v \subset \bigsqcup_{\gamma \in \mathcal{B}} E(\gamma) \setminus \{ \gamma \}$, so $ \Delta_{+}^v \cap \mathcal{B} = \varnothing $. Moreover, as $\Delta_{+}^v$ consists of one element of each Heisenberg pair in $\Delta_+$, its cardinality is equal to the number of Heisenberg pairs, and hence we get
$$ l(v) = \sharp \Delta_{+}^v = \frac{1}{2} \sharp \bigsqcup_{\gamma \in \mathcal{B}} E(\gamma) \setminus \{ \gamma \} = \frac{1}{2} \sharp \left( \Delta_+ \setminus \mathcal{B} \right) = \frac{N-m}{2} $$
where $N = \sharp \Delta_+$ and $m = \sharp \mathcal{B}$.
 Therefore Lemma~\ref{lem : conditions on v} implies that the pair $(v,w)$ with $w := w_0v$ satisfies $v \leq w$ and $l(w)-l(v)=m$. Theorem~\ref{t:kappa-isom} implies that $\Rvw \simeq  ( \cp )^m$. The theorem now folows from Propositions \ref{p:irr-comp} and \ref{prop : maximality}.
 
 \end{proof}
 
\end{document}